\documentclass[11pt,a4paper,DIV=12]{scrartcl}

	\usepackage{mathtools,amssymb,amsthm,mathrsfs,calc,stmaryrd,dsfont}
	\usepackage[british]{babel}
	\usepackage{amsfonts}              
	\usepackage{hyperref}

	\numberwithin{equation}{section}
	\numberwithin{figure}{section}

	\usepackage{todonotes}
	\newtheorem {theorem}{Theorem}[section]
	\newtheorem {proposition}[theorem]{Proposition}
	\newtheorem {lemma}[theorem]{Lemma}
	\newtheorem {corollary}[theorem]{Corollary}

	{\theoremstyle{definition}
		\newtheorem{definition}[theorem]{Definition}
		\newtheorem{example}[theorem]{Example}
		\newtheorem{remark}[theorem]{Remark}
	}

	\allowdisplaybreaks[4]

\DeclareMathOperator{\Id}{\mathrm{I}}

\DeclareMathOperator{\N}{\mathbb{N}}

\DeclareMathOperator{\R}{\mathbb{R}}
\DeclareMathOperator{\E}{\mathds{E}}

\newcommand{\supp}{\mathrm{supp}}

\newcommand{\norm}[1]{\|#1 \|}
\newcommand{\abs}[1]{|#1 |}
\newcommand{\as}{\text{a.s.}}
\newcommand{\skalar}[1]{\langle #1 \rangle}

\newcommand{\RR}{\mathbb{R}}
\newcommand{\QQ}{\mathbb{Q}}
\newcommand{\PP}{\mathds{P}}

\newcommand{\NN}{\mathbb{N}}
\newcommand{\FF}{\mathds{F}}

\newcommand{\EE}{\mathds{E}}

\newcommand{\cF}{\mathcal{F}}

\newcommand{\cA}{\mathcal{A}}
\newcommand{\cE}{\mathcal{E}}

\newcommand{\var}{\mathds{V}\mathrm{ar}\,}

\newcommand{\bone}{\mathds{1}}
\newcommand{\GL}{\operatorname{GL}(\RR,d)}
\newcommand{\SL}{\operatorname{SL}(\RR,d)}
\newcommand{\SO}{\operatorname{SO}(\RR,d)}
\newcommand{\Od}{\operatorname{O}(\RR,d)}
\newcommand{\vect}{\operatorname{vec}}

\newcommand{\adj}{\operatorname{adj}}
\newcommand{\trace}{\operatorname{tr}}

\newcommand\normal{\color{black}}

\begin{document}
	\title{
		Limit theorems for  stochastic exponentials of matrix-valued L\'evy processes \\
	}
	\author{Anita Behme\thanks{Technische Universit\"at
			Dresden, Institut f\"ur Mathematische Stochastik, Helmholtzstra{\ss}e 10, 01069 Dresden, Germany, \texttt{anita.behme@tu-dresden.de}, phone: +49-351-463-32425, fax:  +49-351-463-37251.}\; and Sebastian Mentemeier\thanks{Universität Hildesheim, Institut für Mathematik und Angewandte Informatik, Samelsonplatz 1, 31141 Hildesheim, Germany, \texttt{mentemeier@uni-hildesheim.de}}}
	\date{\today}
	\maketitle

	\begin{abstract} We study the long-time behaviour of matrix-valued stochastic exponentials of Lévy processes, i.e. of multiplicative Lévy processes in the general linear group. In particular, we prove laws of large numbers as well as central limit theorems for the logarithmised norm, logarithmised entries and the logarithmised determinant of the stochastic exponential.
		Where possible, also Berry-Esseen bounds are stated. 
	\end{abstract}



\section{Introduction}

Let $L=(L_t)_{t\geq 0}$ be a Lévy process in $\RR^{d\times d}$. We are interested in the long time behaviour of the \emph{stochastic exponential} of $L$, that is of the unique, càdlàg solution $X=(X_t)_{t\geq 0}$ in $\RR^{d\times d}$ of the stochastic differential equation (SDE)
\begin{equation}\label{SDE}
	dX_t=X_{t-} dL_t,\; t\geq 0, \quad X_0=\Id,
\end{equation}
with $\Id$ the identity matrix.\\
Dating back to results of Skorokhod, cf. \cite{skorokhod82} (see also \cite[Prop. 2.4]{Behme2012}), it is known that under the non-singularity assumption 
\begin{equation}\label{eq-nonsingular}
	\det(\Id + \Delta L_t) \neq 0 \quad \text{for all } t\geq 0, 
\end{equation}
the stochastic exponential of $L$ is a \emph{left Lévy process} in the general linear group $\GL$, i.e. a multiplicative stochastic process taking values in $\GL$, starting in $\Id$, and with stationary and independent increments $X_{s,t}:= X_s^{-1} X_t$, $s\leq t$, acting from the right. Hereby $\Delta L_t=L_t- L_{t-}$ denotes the jump of $L$ at time $t\geq 0$. Conversely, for any left Lévy process $X$ in $\GL$ there exists an additive Lévy process $L$ in $\RR^{d\times d}$ such that $X$ is the stochastic exponential of $L$. In this case $L$ is given by the unique solution of the SDE
\begin{equation} \label{SDElog}
	dL_t = X_{t-}^{-1} dX_t,\; t\geq 0, \quad L_0=0,
\end{equation}
with $0\in \RR^{d\times d}$ denoting the zero matrix, and we call $L$ the \emph{stochastic logarithm} of $X$. 

While in the one-dimensional case a closed-form expression for the stochastic exponential is well-known, cf. \cite[Thm.~II.37]{protter}, for $d>1$ analogue statements are only available under certain commutativity assumptions, see \cite{Yan2012, Kern2022}. This makes the study of stochastic exponentials for general Lévy processes in dimensions $d>1$ much more involved and motivates our investigation of limit theorems.\\

In this work, we prove strong laws of large numbers (SLLN) and central limit theorems (CLT), where possible with Berry-Esseen bounds, for the most relevant processes derived from $X$, namely its norm, its determinant and the growth rate of a fixed entry.  Considering the action of $X$ on the projective space, we further obtain uniqueness of an invariant measure together with convergence rates for the corresponding Markov chain given by the projection of $(yX_t)_{t \ge 0}$ on the projective space. Throughout, the needed assumptions will be formulated in terms of the driving Lévy process $L$.

Our methods will connect the field of multivariate (continuous-time) L\'evy processes with the field of (discrete-time) products of random matrices. Concerning convergence results for the norm, the entries and the associated Markov chain, we will draw on recent results from the theory of products of random matrices which we can apply to discrete-time skeletons of $X$. We will require weak assumptions: irreducibility of the action of $X$ on $\R^d$, and an assumption that basically only excludes cases where $X$ is contained in the set of similarity matrices, i.e. products of dilations and isometries. Considering the determinant process, the analysis will be based on the study of multiplicative L\'evy processes and multivariate Itô formula.

\subsection*{Related works}

Multiplicative Lévy processes in $\GL$, and - even much more general - in Lie groups, are treated in depth in the textbook \cite{Liao2004}. This source also states CLTs for these processes (under stronger assumptions than imposed here), but does not consider the stochastic logarithm $L$ of a multiplicative Lévy process $X$. The SDE \eqref{SDE} of the stochastic exponential for processes with values in arbitrary Lie groups has also been studied by Estrade~\cite{Estrade}. \\
Concerning the special case of stochastic exponential functionals in $\GL$ that we consider in this paper, results on the short-time behaviour can be found in \cite{Reker23}, where in particular the equivalence
$$\lim_{t\downarrow 0} \frac{L_t}{t^p} = \nu \; \text{a.s.} \quad \Leftrightarrow \quad \lim_{t\downarrow 0} \frac{X_t-\Id}{t^p} = \nu,$$
for $\nu \in\RR^{d\times d}$, is proven \cite[Example 1]{Reker23}. As also mentioned in that source, further results concerning short- and long-time behaviour could be derived using the generalizations of Blumenthal-Getoor indices for Feller processes, see \cite{Schilling98, schnurrdiss, Schnurrhdwj}. This approach yields law-of-iterated-logarithm-type results applied on the norm of the process, i.e. it allows to obtain information on $\lim_{t\to \infty} \sup_{s\leq t} f(t)^{-1} \norm{X_s-\Id}$ for suitable growth rate functions $f$. In this article we will not follow this approach via the supremum process, but instead derive SLLNs and CLTs  for the norm process and other derivates of the stochastic exponential without considering its supremum. In parts, namely when considering the determinant process, this can be obtained by applying corresponding results for Lévy processes as they can be found in \cite[Chapter 48]{sato2nd} and \cite{DoneyMaller02}. Concerning other derivatives of the stochastic exponential, we rely on the theory of random matrices. For products of random matrices, the strong law of large numbers for the norm was shown first in \cite{Furstenberg1960}, where also a central limit theorem for matrices with strictly positive entries was provided. For general invertible matrices, a central limit for the norm was proved first in \cite{LePage1982}, building on earlier work by \cite{Kaijser1978}, see \cite{Bougerol1985} for a textbook account of these works. Recent progress was in particular concerned with eliminating the assumption of an exponential moment and proving the optimal second moment condition for the validity of the CLT. This was accomplished in \cite{Benoist2016}, see also the textbook \cite{Benoist2016a}. In the last few years, many more classical limit theorems, like the local limit theorem, large deviation principles and Edgeworth expansions, have been established under (almost) optimal moment conditions, see e.g. \cite{Xiao2023,Xiao2023a,Grama2022,Cuny2023}.

\subsection*{Main results and structure of the paper}

In Section \ref{sect:Preliminaries} we introduce basic notation and provide details about matrix-valued stochastic exponentials and left L\'evy processes in $\GL$. In Section \ref{sect:limitNorms} we discuss the geometric assumptions and the moment assumption required for the validity of limit theorems for norms, entries and the action on a fixed vector. We show how these assumptions can be checked in terms of the characteristics of $L$.
Theorem \ref{CLT} provides the SLLN and the CLT for both $\log \norm{X_t}$ and $\log \norm{yX_t}$, for any $y \in \R^d$, $\norm{y}=1$; under optimal moment assumptions. A CLT for $\log \abs{X^{(i,j)}}$ is given in Theorem~\ref{CLTcoefficients}. Theorem~\ref{BerryEsseen} contains Berry-Esseen bounds for the CLT for $\log \norm{yX_t}$ and, moreover, for the joint convergence of both the directional and radial part of $yX_t$.	
In Section \ref{sect:determinant}, we prove an explicit representation of the determinant process $\det(X_t)$ as the (one-dimensional) stochastic exponential of a L\'evy process $\check{L}_t$, the characteristics of which are given in Theorem~\ref{thm-detprocess}. From this, we derive a SLLN for the logarithm of the determinant process in Theorem~\ref{thm-SLLN-det} as well as a CLT in Theorem~\ref{thm-CLT-det}.

\section{Preliminaries}\label{sect:Preliminaries}

For easier reference  all notation that will be used in the paper is collected in this section.

\subsection{Notation}

We reserve capital letters for random variables (and the identity matrix $\Id$). Deterministic matrices are usually denoted by letters $a$ or $g$. For a matrix $a\in \RR^{m\times n}$ we write
$a^T$ for its transpose and let $a^{(i,j)}$ denote the component
in the $i$-th row and $j$-th column of $a$. The symbol $\otimes$ denotes the Kronecker product. We fix a vector norm $\| \cdot \|$ on $\R^d$ and denote the corresponding operator norm on matrices by $\| \cdot \|$ as well.
The general linear group of degree $d$ over $\RR$ is denoted as $\GL$, while its subgroups, the special linear group, the orthogonal group and the special orthogonal group are denoted as $\SL$, $\Od$ and $\SO$, respectively.
We abbreviate $M(a)= \max\{ \norm{a}, \norm{a^{-1}}\}$ for $a \in \GL$. For $\epsilon>0$, $B_\epsilon(a):= \{g \, : \, \norm{g-a}< \epsilon\}.$

Denote by $P(\RR^d)$ the projective space of dimension $d-1$. For $y \in \R^d \setminus\{0\}$, we denote by $\hat{y} = \R y$ the corresponding direction in $P(\R^d)$.
For a matrix $g$ and $\hat{y}\in P(\RR^{d})$, denote the right action of $g$ on  $\hat y$  by 
$$  \hat y   \cdot g  := \RR (yg) \in P(\RR^d).$$

Let $\mathcal{C}(P(\R^d))$ be the space of continuous complex-valued functions on the projective space $P(\R^d)$. When studying convergence rates, we will have to consider a certain subset of test functions, namely $\gamma$-Hölder continuous functions, for some  $\gamma>0$. We equip the projective space $P(\R^d)$ with the angular distance $\mathbf{d}$, i.e. for any $\hat{y}=\R y$, $\hat{z}=\R z \in P(\R^d)$ we set 
$\mathbf{d}(\hat{y},\hat{z}):=\abs{\sin(\theta(y,z))}$, where $\theta(y,z)$ is the angle between $y$ and $z$. For any $\varphi \in \mathcal{C}(P(\R^d))$ we set
\begin{equation}
	\norm{\varphi}_\gamma := \norm{\varphi}_\infty + [\varphi]_\gamma, \quad \norm{\varphi}_\infty:=\sup_{x \in P(\R^{d-1})} \abs{\varphi(x)}, \quad [\varphi]_\gamma := \sup_{x,y \in P(\R^d)} \frac{\abs{\varphi(x)-\varphi(y)}}{\mathbf{d}^\gamma(x,y)},
\end{equation}
and define  the space  of $\gamma$-Hölder continuous functions by $\mathcal{B}_\gamma:=\big\{ \varphi \in \mathcal{C}(P(\R^d)) \, : \, \norm{\varphi}_\gamma < \infty \big\}$.

Limits in distribution
will be denoted by ``$\operatorname{d-}\lim$'' or ``$\overset{d}\to$'',
limits in probability by ``$\PP\text{-}\lim$'' or
``$\overset{\PP}\to$'', and ``almost surely'' will be abbreviated by
``a.s.''.  The cdf of the standard normal distribution $\mathcal{N}(0,1)$ will be written as $\Phi(x), x\in\RR$. The law of a random matrix $Y$ will be denoted by
$\mathcal{L}(Y)$. 
Jumps of a matrix valued c\`adl\`ag process $X=(X_t)_{t\geq
	0}$ will be denoted by $\Delta X_t := X_t - X_{t-}$ with $X_{t-} :=
\lim_{s\uparrow t} X_s$ for $t>0$ and the convention $X_{0-} := 0$.

 Given a filtration $\FF = (\cF_t)_{t\geq 0}$ satisfying the usual
hypotheses of right-continuity and completeness (see e.g.~\cite[p.~3]{protter}), a matrix-valued stochastic process $M=(M_t)_{t\geq 0}$ is called an $\FF$-semimartingale or simply a semimartingale if every component $(M_t^{(i,j)})_{t\geq 0}$ is a semimartingale with respect to the filtration $\FF$. For a semimartingale $M$ in $\RR^{m\times n}$ and a locally bounded
predictable process $H$ in $\RR^{\ell\times m}$ the $\RR^{\ell\times
	n}$-valued (left) stochastic integral $I=\int HdM$ is given by
$I^{(i,j)}=\sum_{k=1}^m \int H^{(i,k)}dM^{(k,j)}$ and in the same
way for $M \in \RR^{\ell\times m}$, $H\in \RR^{m\times n}$, the $\RR^{\ell
	\times n}$-valued stochastic (right) integral $J=\int dM H$ is given
by $J^{(i,j)}=\sum_{k=1}^m \int H^{(k,j)}dM^{(i,k)}$.

\subsection{Additive and multiplicative L\'evy processes}

Recall that an (additive) L\'evy process $L= (L_t)_{t\geq 0}$ with
values in $\RR^{d\times d}$ is a process with stationary and
independent (additive) increments which has almost surely c\`adl\`ag
paths and starts at $0$. Here, an increment of $L$ is given
by $L_t - L_s$ for $s\leq t$. We will say that an $\RR^{d\times d}$-valued Lévy process $L$ has the \emph{characteristic triplet} $(\Sigma_L, \gamma_L, \nu_L)$ if the $\RR^{d^2}$-valued Lévy process $\vec{L}:=\vect(L)$ has the characteristic triplet $(\Sigma_L, \vect(\gamma_L), \vect(\nu_L))$, where $\Sigma_L$ is a symmetric, $\RR^{d^2 \times d^2}$-valued non-negative definite matrix, $\gamma_L\in \RR^{d^2}$ is the location parameter, $\nu_L$ denotes the Lévy measure of $L$, and $\vect(\nu_L)(B) := \nu_L(\vect^{-1}(B))$ for any Borel set $B\subset\RR^{d^2}$. For easier reference we also write $\sigma_{(m,j),(n,\ell)}$ for the covariance between the Gaussian components of $L^{(m,j)}$ and $L^{(n,\ell)}$, i.e. the $(d(j-1)+m, d(\ell-1)+n)$-entry of $\Sigma_L$. We also use the notation $\gamma_L^0$ for the drift of $L$ whenever it exists and recall that $\gamma_L^0=\gamma_L- \int_{\|x\| \leq 1} x \nu_L(dx)$, cf. \cite[Rem. 8.4]{sato2nd}.

We refer to Sato~\cite{sato2nd} as standard reference for further
information regarding L\'evy processes. \\

Following \cite{Liao2004}, a c\`adl\`ag process $X=(X_t)_{t\geq 0}$ in
$\GL$, with $X_0=I$ a.s. is called a \emph{(multiplicative) left L\'evy process}, if it has independent and stationary \emph{right increments}. Hereby, (multiplicative) right increments are of the form
$$X_{s,t}=X_s^{-1}X_t, \quad \text{for }0\leq s\leq t<\infty,$$
and hence the process $X$ in $\GL$
has \emph{independent right increments} if for any $n\in \NN$,
$0<t_1<\ldots<t_n$, the random variables $X_0,
X_{0, t_1},\ldots, X_{t_{n-1}, t_n}$ are independent, and it has \emph{stationary right increments} if
$X_{s,t} \overset{d}= X_{0,t-s}$ holds for all $s<t$.\\

As already described in the Introduction, under the non-singularity assumption \eqref{eq-nonsingular}
the stochastic exponential $X$ of an additive Lévy process $L$ in $\RR^{d\times d}$ is a left Lévy process in $\GL$. Conversely, any left Lévy process $X$ in $\GL$ has a stochastic logarithm $L$ which is an (additive)  Lévy process in  $\RR^{d\times d}$ and defined as the unique solution of \eqref{SDElog}.
For such a pair $L$ and $X$ it has been shown in \cite[Prop. 3.1]{Behme2012b} that for any $\kappa>0$ 
\begin{equation}\label{eq-momentcondition}
	\EE\Big[\|L_1\|^\kappa\Big]<\infty \quad \text{implies} \quad \EE\Big[\sup_{0\leq s \leq t} \| X_s\| ^\kappa \Big]<\infty \quad \text{for all }t\geq 0.
\end{equation}
In particular 
\begin{equation*}
	\EE[X_t]=\exp\big(t \EE[L_1]\big) \quad \text{for all }t\geq 0.
\end{equation*}

\begin{remark}\label{rem-leftright} 
Clearly, by symmetry, one can also define \emph{right Lévy processes} in $\GL$ and it can be easily veryfied that the inverse and the transpose of a left L\'evy process in $\GL$ are right L\'evy processes and vice versa. \\
Moreover, for any left Lévy process $X$ solving \eqref{SDE}, under the non-singularity assumption \eqref{eq-nonsingular}, the right Lévy process $(X_t^{-1})_{t\geq 0}$ is the solution of the SDE
\begin{equation} \label{SDEright} dX_t^{-1} =dU_t \,X_{t-}^{-1},\;  t\geq 0, \quad X_0^{-1}=\Id,\end{equation} 
for an (additive) Lévy process $U=(U_t)_{t\geq 0}$ in $\RR^{d\times d}$, which fulfils
\begin{equation}\label{eq-LtoU}
	U_t=-L_t-[L,U]_t,\quad t\geq 0,
\end{equation}
see \cite{karandikar91b, Behme2012} for details and proofs.\\
These relations allow for an easy translation of the results on solutions of \eqref{SDE} to results on solutions of the associated SDE \eqref{SDEright}. 
\end{remark}

For a better understanding of the relation between jumps of the driving Lévy process and jumps of the stochastic exponential,  we provide the following result. The formula is also indicated in \cite[Eq. (6)]{skorokhod82}, where however it is stated wrongly and without proof. We therefore also provide a sketch of the proof.
\begin{lemma} \label{lemma-formulaskorokhod}
	Let $L$ be a L\'evy process in $\RR^{d\times d}$ fulfilling \eqref{eq-nonsingular} and let $X$ be the corresponding stochastic exponential. Fix $\varepsilon>0$ and define the Lévy process $L^\varepsilon$ by 
	$$L_t^\varepsilon:= L_t- \sum_{\substack{0<s\leq t\\ \|\Delta L_s\|\geq \varepsilon}} \Delta L_s, \quad t\geq 0,$$
	and denote the stochastic exponential of $L^\varepsilon$ by $X^\varepsilon$. Then, with $N_t(\varepsilon)$ and $0<\tau_1<\ldots<\tau_{N_t(\varepsilon)} \leq t$ denoting the number and times of the jumps of $L_t-L_t^\varepsilon$, respectively, 
	\begin{align*} X_t
		&=  X_t^{\varepsilon} + \sum_{\ell=1}^{N_t(\varepsilon)} (-1)^{\ell+1} \hspace{-5mm} \sum_{1\leq k_1< \ldots<k_\ell\leq N_t(\varepsilon)} \hspace{-5mm} X_{\tau_{k_1}-}  \Delta L_{\tau_{k_1}}  X_{\tau_{k_1}, \tau_{k_2}-} \cdots \Delta L_{\tau_{k_\ell}} X_{\tau_{k_\ell}, t}
	\end{align*} 
\end{lemma}
\begin{proof}	
	First of all note that $N_t(\varepsilon)<\infty$ a.s. as $L$ has only finitely many big jumps. Further, by definition
	\begin{align*}
		X_t^\varepsilon &
		= X_{\tau_1-} X_{\tau_1, \tau_2-} \cdots  X_{\tau_{N_t}(\varepsilon), t}, \\
		\text{while} \quad X_t&
		=  X_{\tau_1-} X_{\tau_{1}-,\tau_{1}} X_{\tau_1, \tau_2-} X_{\tau_{2}-,\tau_{2}} \cdots  X_{\tau_{N_t}(\varepsilon), t}\\
		&= X_{\tau_1-} (\Id + \Delta L_{\tau_1}) X_{\tau_1, \tau_2-}  (\Id + \Delta L_{\tau_2}) \cdots  X_{\tau_{N_t}(\varepsilon), t}.
		\end{align*} 
	With this the given formula can be checked by induction on $N_t(\varepsilon)$  and direct computation.
\end{proof}

We also mention the following result due to \'Emery~\cite{emery}.
\begin{lemma} \label{lemma-emeryapprox} Let $\sigma=(t_0=0,t_1,\ldots, t_j, \ldots)$ with $t_j\to
	\infty$ and $|\sigma|:=\sup_{j\in \NN} |t_j-t_{j-1}|<\infty$ be a
	subdivision of the positive real line. Let $L$ be a L\'evy process
	in $\RR^{d\times d}$ with \eqref{eq-nonsingular}. Then the processes $X^\sigma$ given by
	$X^\sigma_0 := \Id$ and
	\begin{equation} \label{MGOUemeryapprox}
		X^\sigma_t := (\Id+L_{t_1})(\Id+L_{t_2}-L_{t_1})\cdots
		(\Id+L_{t_j}-L_{t_{j-1}})(\Id+L_t-L_{t_j})
	\end{equation}
	for $t_j<t\leq t_{j+1}$ converge to the stochastic exponential $X$ of $L$ uniformly on compacts
	in probability when $|\sigma|$ tends to $0$. 
\end{lemma}

In order to make results for Feller processes accessible in the study of matrix-valued stochastic exponentials, we also provide the generator of the stochastic exponential in the next proposition. Its proof is deferred to Appendix \ref{appendix:stochastic.exponential.as.Feller.process}.

	\begin{proposition}\label{prop:generatorXt}
		Let $L=(L_t)_{t\geq 0}$ be a Lévy process in $\RR^{d\times d}$ with characteristics $(\Sigma_L,\gamma_L,  \nu_L)$ fulfilling \eqref{eq-nonsingular}, and let $X=(X_t)_{t\geq 0}$ be its stochastic exponential. Then $X$ is a rich Feller process in $\RR^{d\times d}$ and  its extended generator is given by
		\begin{align*}\cA_{X} f(x) &= \sum_{i,j=1}^d \ell(x)^{(i,j)} \frac{\partial}{\partial x^{(i,j)}} f(x) + \frac12 \sum_{i,j, k,l=1}^{d} Q(x)^{(i,j,k,l)} \frac{\partial}{\partial x^{(i,j)}}\frac{\partial}{\partial x^{(k,l)}} f(x) \\
			&\quad  + \int_{y\in \RR^{d\times d}\setminus \{0\}} \Big( f(x+y) - f(x) -  \sum_{i,j=1}^d y^{(i,j)} \frac{\partial}{\partial x^{(i,j)}} f(x)\mathds{1}_{\{\norm{y}\leq 1\}}  \Big) N(x, dy), \end{align*}
		for all $f\in C_c^2(\RR^{d\times d})\subseteq D(\mathcal{A}_{X})$, the latter denoting the domain of $\cA_X$. For any $x\in \RR^{d\times d}$, the triplet $(\ell(x), Q(x), N(x,dy))$ is given by
		\begin{align*}
			\ell(x)&=  x \gamma_L +  \int_{y\in \RR^{d\times d}\setminus\{0\}}  x y  \left(  \mathds{1}_{\{ \norm{ x y } \leq 1\}}- \mathds{1}_{\{ \norm{y} \leq 1\}} \right) \nu_L(dy),  \\
			Q(x)^{(i,j,k,l)} &= \sum_{n,m=1}^{d} 	x^{(i,m)} \sigma_{(m,j),(n,l)} x^{(k,n)} , \, i,j,k,l=1,\ldots, d,\quad \text{and} \\
			N(x,dy)&= \nu_L(T_x^{-1}(dy)),
		\end{align*}
		being the image measure of $\nu_L$ under the transform $T_x$ defined by
		$$T_x(y): \RR^{d\times d} \to \RR^{d\times d} , y\mapsto x y, \quad x\in\RR^{d\times d}.$$
\end{proposition}

\subsection{Induced processes}

Let $L=(L_t)_{t\geq 0}$ be a Lévy process in $\RR^{d\times d}$ that fulfils \eqref{eq-nonsingular}, and let $X=(X_t)_{t\geq 0}$ denote its stochastic exponential. We define the \emph{determinant process} $D=(D_t)_{t\geq 0}$ of $X$ by setting
\begin{equation}\label{eq-detprocess} D_t=\det(X_t), \quad t\geq 0. \end{equation}
As will be shown in Section \ref{sect:determinant}, this process is again a multiplicative Lévy process and hence a stochastic exponential.\\
Further, for each fixed $x \in \RR^d$, we consider the Markov process 
\begin{equation}\label{eq-onepoint}
	Y_t^x:=x X_t 
\end{equation}
on $\RR^d$, called {\em one-point motion}.
By \cite[Prop. 2.1]{Liao2004} the one-point motion $(Y_t^x)_{t\geq 0}$ also is a Feller process on $\RR^d$. Moreover, as
\begin{align*}
	Y_t^x =	xX_t &= x\left(I + \int_{(0,t]} X_{s-} dL_s\right)\\
	&= x + \int_{(0,t]} x X_{s-} d L_s = x + \int_{(0,t]} Y_{s-}^x dL_s,
\end{align*}
we observe that $Y_t^x$ also solves \eqref{SDE} with different starting value and state space. Still, due to its state space $\RR^d$ this process is not a left Lévy process. 
The projection of $Y_t^x$ on the projective space is denoted
\begin{equation}\label{eq-onepoint-projective}
	 Z_t^x := x \cdot X_t:=\RR Y_t^x
\end{equation}
and constitutes a Markov process as well. 
In the same way, for each $h>0$, the discrete time process
\begin{equation}\label{eq-onepoint-projective-discrete}
	 Z_{nh}^y:= y \cdot X_{nh}, \qquad n \in \N,
\end{equation}
constitutes a Markov chain on $P(\RR^d)$. If $Z_0$ is an arbitrary (random) value, we omit the superscript and write $Z_t=Z_0 \cdot X_t$, $t\ge 0$ as well as $Z_{nh}=Z_0 \cdot X_{nh}$, $n \in \N$.

\begin{remark}
	The Markovian structure of $Z_n^x$ is due to the fact that new information enters from the right, 
	$$Z_n^x=x \cdot X_n=(x \cdot X_k) \cdot X_{k,n}=Z_k^x \cdot X_{k,n}.$$
	In the theory of iterated function systems, $Z_n^x$ is called a {\em forward} process, while  the action from the left on a vector, $X_n \cdot x= X_k \cdot (X_{k,n} \cdot x)$ is called a {\em backward process}, see \cite{Diaconis1999} for more information and details. 
\end{remark}

	\subsection{Supports and subgroups}
	
	 Let $X=(X_t)_{t\geq 0}$ be the stochastic exponential of some Lévy proces $L=(L_t)_{t\geq 0}$ in $\RR^{d\times d}$ that fulfils \eqref{eq-nonsingular}. Let $E$ be an exponential random variable with parameter one that is independent of $X$. Introducing the probability measure $\mu:=\PP(X_E \in \cdot)$, we denote by $G_X$ the subgroup of $\GL$ generated by the support $\supp(\mu)$ of $\mu$. In other words, $G_X$ is the smallest closed subgroup of $\GL$ that contains $X_E$, and hence $X_t$ for each $t>0$, with probability 1. 
	Note that, upon writing $\mu_t=\PP(X_t \in \cdot)$, we have
	$ \mu(A)=\int_0^\infty e^{-t} \mu_t(A) dt$
	for all Borel sets $A \subset \GL$. 
	
	In a similar way, we denote for each $h>0$ by $G^h_X$ the smallest closed subgroup of $\GL$ that contains $X_{nh}$, ${n \in \N}$ with probability 1.
	Further, we denote by $T_X$ and $T_X^h$ the smallest closed subsemigroups of $\GL$, containing $X_t$, ${t \ge 0}$ and $X_{nh}$, ${n \in \N}$, respectively, with probability 1. 
	\normal

\section{Limit theorems for norm, entries and the one-point motion}\label{sect:limitNorms}

 Throughout this section let $L=(L_t)_{t\geq 0}$ be a Lévy process in $\RR^{d\times d}$ that fulfils \eqref{eq-nonsingular}, and let $X=(X_t)_{t\geq 0}$ denote its stochastic exponential.

In this section, we prove a strong law of large numbers and a central limit theorem with Berry-Esseen bounds for the logarithm of both the norm and the entries of the left L\'evy process $X$ in $\GL$. We will make use of corresponding limit theorems for (discrete-time) products of random matrices in $\GL$ as provided in \cite{Benoist2016,Xiao2022,Cuny2023}. 
These limit theorems will require, in addition to natural moment assumptions, some geometric assumptions like irreducibility of the action of $X$ on $\R^d$, which we are going to study in the first part of this section.

Note that in most of the literature on products of random matrices that we quote, e.g. \cite{Guivarch2016,Benoist2016,Grama2022},  products with left increments are considered, that is, products of the form
$$ X_{n,n-1} \cdots X_{1}.$$
Obviously, by taking the transpose, we can reverse the order of multiplication, and transfer results for left products to results for right products, and vice versa.

\subsection{The geometric assumptions}

In this subsection, we formulate the geometric assumptions on $X$ and show how they  can be checked directly on the driving L\'evy process $L$. These assumptions will be formulated in terms of the group $G_X$ generated by (the support) of $X$.

\begin{lemma}\label{lem:Gmu.as.union}
	 The group $G_X$ equals the group generated by
	$$ \bigcup_{t >0} \supp(\mu_t).$$
\end{lemma} 

\begin{proof}[Source]
	This is \cite[Prop. 6.7]{Liao2004}.
\end{proof}
In the same way, we define $T_X$ to be the semigroup generated by $\supp(\mu)$. 
We will impose the following irreducibility and proximality conditions on $G_X$, cf. \cite{Guivarch1986,Bougerol1985}.

	A (semi-)group $T$ acts {\em strongly irreducibly} on $\R^d$ (from the right), if there is no finite union $\mathcal{W}=\bigcup_{i=1}^n W_i$ of proper subspaces $W_i\subsetneq \RR^d$, satisfying $\mathcal{W}T \subset \mathcal{W}$. 
	Note that by the discussion after Notation III.2.2 in \cite{Bougerol1985}, strong irreducibility of a semigroup $T$ is equivalent to strong irreducibility of the group $G$ generated by $T$. In particular, if $T=T_X$, then strong irreducibility of $T_X$ is equivalent to strong irreducibility of $G_X$.

A matrix $a \in \GL$ is called {\em proximal}, if it has an algebraically simple dominant eigenvalue, the modulus of which exceeds that of all other eigenvalues (comparable to the Perron-Frobenius eigenvalue).
	A matrix (semi-)group $T$ is called {\em proximal}, if it contains a proximal element.
Obviously, proximality of $T_X$ implies proximality of $G_X$.

\begin{definition}
	A matrix (semi-)group  $T$ satisfies condition (i-p), if $T$ acts strongly irreducibly on $\R^d$ and is proximal. 
\end{definition}

We want to obtain sufficient conditions for the strong irreducibility and proximality of $G_X$ in terms of the driving L\'evy process $L$. As already indicated by Lemma \ref{lemma-emeryapprox}, properties of $G_X$ are governed by the small-time behaviour of products of the form $(\Id+L_h)$. This is why we cannot provide a result that compares the (additive!) group generated by the support of $L$ with that of the (multiplicative) group generated by the support of $X$. 

Our first result shows that $G_X$ satisfies (i-p) whenever $L$ has a "full" Brownian component.

\begin{proposition}\label{prop_irreducible}
	Assume $L$ is a Lévy process in $\RR^{d\times d}$ fulfilling \eqref{eq-nonsingular} with characteristic triplet $(\Sigma_L, \gamma_L, \nu_L)$ such that $\Sigma_L$ is positive definite. Then, for all $t>0$, the semigroup generated by $\supp(\mu_t)$ satisfies condition (i-p), hence the same holds true for $T_X$ and $G_X$.
\end{proposition}
\begin{proof}
	As $\Sigma_L$ is positive definite, it follows
	that the density of the Brownian component is positive on the whole of $\R^{d \times d}$ and hence $\supp(L_h)=\RR^{d\times d}$ for all $h>0$. 
	For $X_t$ consider the approximation as in Lemma \ref{lemma-emeryapprox}
	\begin{align*}
		X_t^\sigma 
		&= (\Id+L_{t_1}) (\Id+L_{t_2}-L_{t_1}) \cdots (\Id+L_{t_j} - L_{t_{j-1}}) ( \Id+L_t-L_{t_j}).
	\end{align*}	
	By the above, the support of this finite product contains any open set in $\GL$, independent of the choice of $\sigma$. This implies the statement via \cite[Prop. IV.2.3]{Bougerol1985}.
\end{proof}

In order to capture Lévy processes without Brownian component, we next consider the case that $L$ is a compound Poisson process.

\begin{proposition}\label{prop_irreducible:CPPdrift}
	Assume $L$ is a Lévy process in $\RR^{d\times d}$ fulfilling \eqref{eq-nonsingular} with characteristic triplet $(\Sigma_L, \gamma_L, \nu_L)$ such that $\Sigma_L=0$ and $\nu_L$ has finite total mass, i.e. $L$ is a compound Poisson process, possibly with drift $\gamma_L^0 \in \RR^{d\times d}$. Assume that the closed semigroup generated by
	\begin{equation}\label{eq:CPPsemigroup}
		\big\{ e^{t \gamma_L^0} \, : \, t >0 \big\} \cup \big(\Id+\supp(\nu_L)\big)
	\end{equation}
	satisfies condition (i-p). Then the same holds true for $G_X$.
	
	In particular,  if the closed semigroup generated by $\big(\Id+\supp(\nu_L)\big)$ satisfies condition (i-p), then so does $G_X$.
\end{proposition}

\begin{proof}
First, observe that for any fixed $t>0$, with positive probability, $X_t=e^{t \gamma_L^0}$,  i.e. no jump occurs. Secondly, with positive probability we have exactly one jump at the random time $\tau_1$, in which case
\begin{align*}
	X_{t}~&=~ \Id \cdot e^{\tau_1 \gamma_L^0} \cdot (\Id+ \Delta L_{\tau_1}) \cdot e^{(t-\tau_1)\gamma_L^0}.
\end{align*}
By Lemma \ref{lem:Gmu.as.union}, $G_X$ equals the group generated by $\bigcup_{t >0} \supp (\mu_t)$; in particular, $e^{s \gamma_L^0} \in G_X$ for every $s>0$. In addition, $\supp(\mu_t)$ contains for any possible jump $\Delta L$ of $L$  elements of the form $e^{s \gamma_L^0} (I+\Delta L) e^{(t-s) \gamma_L^0}$. Hence, using that $G_X$ is a group, $(\Id+\Delta L) \in G_X$.
We have thus shown that the set \eqref{eq:CPPsemigroup} is contained in the group $G_X$, hence the same is true for the closed semigroup generated by $\eqref{eq:CPPsemigroup}$. 
\end{proof}

\begin{example}
	Consider a compound Poisson process $L$ in $\RR^{2 \times 2}$ with drift $\gamma_L^0$ and jump measure $\nu_L:=\delta_a$, given by 
	$$ \gamma_L^0 = \begin{pmatrix}
		0 & -1 \\
		1 & 0
	\end{pmatrix}, \qquad a = \begin{pmatrix}
		1 & 0 \\
		0 & 0
	\end{pmatrix}$$
	Then the closed semigroup generated by \eqref{eq:CPPsemigroup} satisfies condition (i-p).
\end{example}

\begin{proof}
	The matrix exponential $e^{t \gamma_L^0}$ is a rotation with angle $t$, i.e.
	$$e^{t \gamma_L^0} = \begin{pmatrix}
		\cos(t) & -\sin(t) \\
		\sin(t) & \cos(t)
	\end{pmatrix}.$$
	Hence, $\big\{ e^{t \gamma_L^0} \, : \, t >0 \big\}=\operatorname{SO}(\RR,2)$. The matrix $g=\Id+a$ has two distinct eigenvalues, and it holds that
	$kgk^{-1}$ is in the semigroup for any $k \in \operatorname{SO}(\RR,2)$, using here that the inverse of $k=e^{t \gamma_L^0}$ is $k^{-1}=e^{(2 \pi n-t) \gamma_L^0}$, for $2(n-1)\pi \le t < 2n\pi$.
	This allows us to apply \cite[Prop. 2.5]{Bougerol1985}, which gives that the closed semigroup is strongly irreducible and proximal (1-strongly irreducible and 1-contracting in their notation).
\end{proof}

\begin{example}
	Consider a compound Poisson process $L$ in $\RR^{2 \times 2}$ with drift $\gamma_L^0$ and jump measure $\nu_L:=\delta_a$, given by 
	$$ \gamma_L^0 = \begin{pmatrix}
		1 & 0 \\
		0 & 0
	\end{pmatrix}, \qquad a = 
	\begin{pmatrix}
		\cos(\varphi) & -\sin(\varphi) \\
		\sin(\varphi) & \cos(\varphi)
	\end{pmatrix} - \Id ,
	$$
	for some $\varphi \notin 2\pi \QQ$.
	Then the closed semigroup generated by \eqref{eq:CPPsemigroup} satisfies condition (i-p).
\end{example}

\begin{proof}
	Since $\Id +a$ is an irrational rotation, its powers are dense in $\operatorname{SO}(\RR,2)$. Hence the closed semigroup contains $\operatorname{SO}(\RR,2)$, as well as the proximal matrix 
	$$e^{\gamma_L^0} = \begin{pmatrix}
		e & 0 \\
		0 & 1
	\end{pmatrix}  .$$
	This allows us to apply (as before) \cite[Prop. 2.5]{Bougerol1985} to conclude the assertion.  
\end{proof}

Finally, we consider the case of a L\'evy process with infinite jump activity.  

\begin{proposition}\label{prop_irreducible:CPPdrift:infinite}
	Assume $L$ is a Lévy process in $\RR^{d\times d}$ fulfilling \eqref{eq-nonsingular} with characteristic triplet $(\Sigma_L, \gamma_L, \nu_L)$ such that $\Sigma_L=0$. 
	Assume that there is $\epsilon>0$ such that the closed semigroup generated by
	\begin{equation}\label{eq:noCPPsemigroup}
		\Id + \Big(\supp(\nu_L) \cap B_{\epsilon}(0)^c\Big) 
	\end{equation}
	satisfies condition (i-p). Then the same holds true for $G_X$.
\end{proposition}

\begin{proof}
	We are going to show that $G_X$ contains the set \eqref{eq:noCPPsemigroup}. By Lemma \ref{lem:Gmu.as.union}, $\supp(\mu)$ equals the closure of $\bigcup_{t >0} \supp (\mu_t)$, which is hence contained in $G_X$.
	Let $X_t^\epsilon$ be as in Lemma \ref{lemma-formulaskorokhod}.
	First, observe that for any fixed $t>0$, with positive probability, $X_t=X_t^\epsilon$. Secondly, with positive probability we have exactly one big jump at the random time $\tau_1$, in which case
	\begin{align*}
		X_{t}~&=~ \Id \cdot X_{\tau_1 -}^\epsilon \cdot (\Id+ \Delta L_{\tau_1}) \cdot X_{\tau_1,t}^\epsilon.
	\end{align*}
	Using that $G_X$ is a group, $(\Id+\Delta L) \in G_X$.
\end{proof}

\begin{remark} In Section 6.6 of the book \cite{Liao2004}, sufficient conditions are provided for the two properties (a): $G_X=\GL$ (which obviously implies strong irreducibility) and (b): $G_X$ contains a proximal element (called contracting there). The conditions are formulated in a general setting of L\'evy processes on Lie groups. Comparing the representation of the generator of $X$ given in Prop. \ref{prop:generatorXt}, to the representation of the generator of a general Lévy process on a Lie group in \cite[Eq. (1.7)]{Liao2004} and using the identification of the tangent space of $\GL$ with the set of $d \times d$ real matrices (see Section 1.5 in \cite{Liao2004}), these conditions translate as follows: By \cite[Prop. 6.9]{Liao2004}, if the closure of $\supp(\nu_L)$ equals $\R^{d \times d}$ 
	 then $G_X=\GL$.
	 By \cite[Prop 6.12]{Liao2004}, if $\nu_L$ has finite total mass and the Brownian component of $L$ spans $\R^{d \times d}$, then again $G_X=\GL$.
	
Further, by \cite[Prop. 6.13]{Liao2004}, $G_X$ is proximal, if either $\supp (\nu_L)$ contains an element with $d$ distinct eigenvalues, or  $\nu_L$ has finite mass and the drift part together with the Brownian component of $L$ generate an element with $d$ distinct eigenvalues.
\end{remark}

\begin{example} We finish this part with a counterexample where (i-p) fails to hold.
	
	Assume that $L$ is a L\'evy process in $\R^{d\times d}$ and that there is a proper subspace $W \subsetneq \R^d$ such that $\PP(W L_t \subset W)=1$ for all $t \ge 0$. Then $G_X$ is not irreducible: For any $v \in W$, using the approximation as in Lemma \ref{lemma-emeryapprox}, it follows that  
	\begin{align*}
		X_t^\sigma v 
		&= (\Id+L_{t_1}) (\Id+L_{t_2}-L_{t_1}) \cdots (\Id+L_{t_j} - L_{t_{j-1}}) ( \Id+L_t-L_{t_j}) v 
	\end{align*}
	is again an element of $W$. Since $W$ is a closed set, it follows that $X_t^v = \PP\text{-}\lim_{|\sigma|\to 0} X_t^\sigma v$ is again an element of $W$. It follows that $W G_X \subseteq W$, hence $G_X$ is not irreducible.
	
	Proximality fails e.g. if all elements of $G_X$ are similarity matrices, i.e. products of dilations and isometries.  
\end{example}

To deduce the limit theorems for $X$ from corresponding discrete-time results, we will work with discrete skeletons $(X_{nh})_{n \in \N}$ for fixed $h>0$, and we will require that these skeletons satisfy condition (i-p) as well. More precisely, recall that we denote by $G_X^h$ the smallest closed subgroup of $\GL$ that contains $(X_{nh})_{n \in \N}$ with probability 1. The desired result is given by the subsequent lemma and is a consequence of the right continuity of the paths of $X$.

\begin{lemma}
	Assume that $G_X$ satisfies condition (i-p). Then there is $\ell \in \N$, such that for $h=2^{-\ell}$, the group $G_X^h$ satisfies condition (i-p) as well.
\end{lemma}

\begin{proof}
	We consider first the strong irreducibility. Suppose that for each $\ell$, there is a finite union $\mathcal{W}_\ell$ of proper subspaces such that $ \mathcal{W}_\ell G_X^{2^{-\ell}} \subset \mathcal{W}_\ell$. Whenever $\ell \le m$, it holds $G_X^{2^{-\ell}} \subset G_X^{2^{-m}}$ and hence may choose $\mathcal{W}_\ell$ to be equal to $\mathcal{W}_m$. Hence, there is $\mathcal{W}$ such that $\mathcal{W} G_X^{2^{-\ell}}  \subset \mathcal{W}$ holds for all $\ell \in \N$. Since $\mathcal{W}$ is a closed set and $X$ has right-continuous paths, it then follows that $ \mathcal{W} G_X \subset \mathcal{W}$ as well, contradicting the assumption. Thus, there is $h=2^{-\ell}$ such that $G_X^h$ acts strongly irreducibly on $\R^d$.
	
	If $g$ is a proximal matrix, then by classical perturbation theory for eigenvalues (see e.g. \cite[Thm. 6.3.12]{Horn1990}) there is an open neighbourhood $B(g)$ of $g$ such that all elements of $B$ are also proximal. This first shows that if $g \in G_X$ is proximal, then there is $t>0$ such that $\PP (X_t \in B(g))>0$, and, by right continuity, also $n \in \N$, $h=2^{-\ell}$ such that $\PP(X_{nh} \in B(g))>0$. Hence $G_X^h$ contains a proximal element. 
	
	Since the groups $G^{2^{-\ell}}$ are nested, we may choose the finer skeleton such that both strong irreducibility and proximality are satisfied.	
\end{proof}

\subsection{Moment conditions}

The moment conditions will be formulated in terms of the quantity $M(a)= \max\{ \norm{a}, \norm{a^{-1}}\}$; namely, it will be required that $\EE [ \log M(X_t)]<\infty$ or $\EE [M(X_t)^\epsilon]<\infty$ for some $\epsilon >0$. Since the latter property implies the first, while being easier to check on the driving process $L$, we restrict our attention to the finiteness of small moments of $M(X_t)$. A sufficient condition for this is provided by the following lemma.

\begin{lemma}\label{lem:moment.condition}
	Let $L$ be a Lévy process in $\RR^{d\times d}$ fulfilling \eqref{eq-nonsingular} that satisfies 	\begin{equation} \label{eq-Mexpfinite} \int_{\substack{x\in\RR^{d\times d} \\ \|x\|> 1}} \|x\|^\epsilon \nu_L(dx)<\infty, \text{ and } \int_{\substack{x\in\RR^{d\times d} \\ \|(\Id+x)^{-1} - \Id\|> 1}} \|(\Id+x)^{-1} - \Id\|^\epsilon \nu_L(dx)<\infty, \end{equation} 
	for some $\epsilon>0$. 
	Then $\EE\left[\sup_{0 \le s \le t} M(X_t)^\epsilon\right]<\infty$ for all $t\geq 0$.
\end{lemma}
\begin{proof}
	We have that, for any $t\geq 0$	
	\begin{align}
	\EE\Big[\sup_{0 \le s \le t} M(X_t)^\epsilon\Big]&= \EE\Big[ \sup_{0 \le s \le t} \big(\max\{ \|X_t\|, \|X_t^{-1}\| \} \big)^\epsilon\Big] \nonumber = \EE\Big[\sup_{0 \le s \le t} \max\big\{ \|X_t\|^\epsilon, \|X_t^{-1}\|^\epsilon \big\}\Big] \nonumber \\
		&=  \EE\Big[\sup_{0 \le s \le t} \max\{ \| \cE(L)_t\|^\epsilon, \| \overset{\rightarrow}\cE(U)_t\|^\epsilon \}\Big] \nonumber \\
		&\leq \EE\Big[\sup_{0 \le s \le t} \|\cE(L)_t\|^\epsilon + \sup_{0 \le s \le t} \| \overset{\rightarrow}\cE(U)_t\|^\epsilon\Big] \nonumber \\
		&=  \EE\Big[\sup_{0 \le s \le t} \|\cE(L)_t\|^\epsilon\Big]  +  \EE\Big[\sup_{0 \le s \le t} \| \overset{\rightarrow}\cE(U)_t\|^\epsilon\Big], \label{eq-momentestimate}
	\end{align}
	where $(\overset{\rightarrow}\cE(U)_t)_{t\geq 0}$ denotes the right Lévy process solving \eqref{SDEright} and $U=(U_t)_{t\geq 0}$ solves \eqref{eq-LtoU}. \\
	It is a well-known fact from the theory of Lévy processes, that, for all $\epsilon >0$ the first assertion in \eqref{eq-Mexpfinite} is equivalent to 	$\EE[\|L_t\|^\epsilon]<\infty$ for all $t\geq 0$, cf. \cite[Cor. 25.8]{sato2nd}. \\
	Further, it follows from \cite[Thm. 1]{karandikar91b}, that
	$$\Delta U = (\Id+\Delta L)^{-1} - \Id,$$
	and hence the second assumption in \eqref{eq-Mexpfinite} is equivalent to $\EE[\|U_t\|^\epsilon]<\infty$ for all $t\geq 0$. Thus, via \eqref{eq-momentcondition} and its counterpart for right Lévy processes as stated in \cite[Prop. 3.1]{Behme2012b}, we conclude that \eqref{eq-momentcondition} implies finiteness of \eqref{eq-momentestimate} and hence the statement.
\end{proof}

\begin{remark}
	The assumptions of Lemma \ref{lem:moment.condition} are obviously satisfied if $L$ has only a Brownian component with drift. 
	
	If the support of $\nu_L$ is contained in $B_1(0)$, clearly the first condition in \eqref{eq-momentcondition} is fulfilled while for the second assumption we obtain the following simpler condition. Using that for $x \in \R^{d \times d}$ with $\norm{x}<1$,
	$$ (\Id-(-x))^{-1}=\sum_{k=0}^\infty (-1)^k x^k,$$
	we obtain for such $x$ that 
	$$\norm{(\Id +x)^{-1}-\Id} \le \sum_{k=1}^\infty \norm{x}^k \le \frac{1}{1-\norm{x}}.$$
	Hence, in this case, \eqref{eq-Mexpfinite} is satisfied if
	$$ \int_{\substack{x\in\RR^{d\times d} \\ 0<\|x\|< 1}} \Big(\frac{1}{1-\norm{x}}\Big)^\epsilon \, \nu_L(dx) < \infty.$$
\end{remark}

The role of the quantity $M(a)$ in our derivation is clarified by the subsequent two lemmas, that will be used to show that under appropriate moment assumptions, we may neglect small-time contributions.

\begin{lemma}\label{lem:EMa}
	Let $g\in\RR^{n\times d}$ be either a vector or a matrix with $\norm{g}=1$, and let $a \in \GL$.
	Then
	\begin{equation}\label{eq:bound for log norm ga}
		\big|\log\norm{ga} \big| \le \log M(a).
	\end{equation}
\end{lemma}

\begin{proof}
	Using properties of the operator norm, we have
	$$1=\norm{g} = \norm{g a a^{-1}} \le \norm{ga} \norm{a^{-1}} \quad \Leftrightarrow \quad \norm{ga} \ge \frac{1}{\norm{a^{-1}}}$$
	and hence 
	$$ - \log \norm{a^{-1}} = \log \Big(\norm{a^{-1}}^{-1}\Big) \le \log \norm{ga} \le \log \norm{a}.$$
	It follows that
	$$ \big| \log \norm{ga}\big| \le \log \max\{ \norm{a}, \norm{a^{-1}}\},$$
	where we have dropped the absolute value sign in the last term, since at least one of $\norm{a}, \norm{a^{-1}}$ is greater or equal to one.
\end{proof}

\begin{lemma}\label{lem:M(Xs).vanishes}
	Assume that $\E \big[\sup_{0 \le s \le 1} (\log M(X_s))^\kappa\big]<\infty$ for some $\kappa \in (0,2].$ Then
	$$ \lim_{n \to \infty} \frac1{n^{1/\kappa}}  \sup_{0\le s \le1} \log M(X_{n,n+s})  = 0  \quad \as$$
\end{lemma}

\begin{proof}
	Fix $\epsilon>0$. By stationarity of the increments of $X$,
	\begin{align*}
		\infty >\frac{1}{\epsilon^\kappa} \E \Big[ \sup_{0 \le s \le 1} \Big(\log M(X_s)\Big)^\kappa \Big] & \ge \sum_{n=1}^\infty \PP \Big( \sup_{0 \le s \le 1} (\log M(X_s))^\kappa > n \epsilon^\kappa \Big) \\
		& = \sum_{n=1}^\infty \PP \Big( \frac{1}{n^{1/\kappa}}\sup_{0 \le s \le 1} \log M(X_s)) >  \epsilon \Big) \\ 
		& = \sum_{n=1}^\infty \PP \Big( \frac{1}{n^{1/\kappa}}\sup_{0 \le s \le 1} \log M(X_{n,n+s}) >  \epsilon \Big) .
	\end{align*}
	The result follows	by a standard application of the Borel-Cantelli lemma.
\end{proof}

\subsection{The invariant measure on the projective space}

In Equations \eqref{eq-onepoint-projective} and \eqref{eq-onepoint-projective-discrete}, for an arbitrary (random) initial value $Z_0$, we have introduced the Markov processes $Z_t=Z_0 \cdot X_t$, $t\ge 0$, as well as $Z_{nh}=Z_0 \cdot X_{nh}$, $n \in \N$, on the projective space. In this subsection, we will show that $(Z_t)_{t\geq 0}$ has a unique invariant probability measure $\pi$ on the projective space, and we will study rates of convergence; thereby describing the asymptotic behaviour of directions of the one-point motion.

\begin{lemma}\label{lem:uniqueness pih} 
	Let $h>0$. Assume that $G_X^h$ satisfies condition (i-p). Then there is a unique stationary probability measure for $(Z_{nh})_{n \in \N}$, which we denote by $\pi^h$.
\end{lemma}

\begin{proof}[Source] This is 
	\cite[Thm. III.4.3]{Bougerol1985}.
\end{proof}

\begin{proposition} Let $L$ be a L\'evy process in $\RR^{d\times d}$ fulfilling \eqref{eq-nonsingular} and let $X$ be the corresponding stochastic exponential.
	Assume that $G_X$ satisfies condition (i-p).  Then $Z_t$ has a unique invariant probability measure, which we denote by $\pi$. For each $h>0$, it holds $\pi=\pi^h$.
\end{proposition}

\begin{proof}
	As mentioned in Section \ref{sect:Preliminaries}, for any $x \in \RR^d\setminus\{0\}$, by \cite[Prop. 2.1]{Liao2004}, $Y_t^x=x X_t^x$ is a Feller process with state space $\R^d \setminus\{0\}$. Hence,
	$Z_t^x$, the projection of $Y_t^x$ onto $P(\RR^d)$, is a Feller process as well [since we just reduce the set of test functions], now with a compact state space. This allows to apply \cite[Thm. 3.38]{Liggett2010}, which gives that there is at least one invariant measure for $Z_t$.
	
	If now $\pi$ is an invariant measure for $Z_t$, then it is invariant for $(Z_{nh})_{n \in \N}$ as well, for every $h>0$. By Lemma \ref{lem:uniqueness pih}, it follows that $\pi^h=\pi$. This in particular proves the uniqueness of $\pi$.	
\end{proof}

Under an additional moment condition, we may also obtain rates of convergence towards the stationary distribution, more precisely, of $\E f(Z_t^x)$ towards $\pi(f):=\int f(y) \pi(dy)$ uniformly for all $f$ from a class of suitable test functions.

\begin{proposition}\label{prop:convergenceOnepoint} Let $L$ be a L\'evy process in $\RR^{d\times d}$ fulfilling \eqref{eq-nonsingular} and let $X$ be the corresponding stochastic exponential.
		Assume that $G_X$ satisfies condition (i-p) and $\EE \big[\sup_{0 \le s \le 1} M(X_s)^\epsilon \big]<\infty$ for some $\epsilon >0$. 
	Then there are $\gamma>0$, $0<c<1$ and $n_0 \in \N$, such that for all $n\ge n_0$ and all $f \in \mathcal{B}_\gamma$,
	\begin{align}\label{eq:onepoint.theorem1}
	\sup_{y,z \in P(\RR^d)} \sup_{0 \le r \le 1}	\Big|\EE \big[f(Z_{n+r}^y)\big] - \EE \big[f(Z_n^z)\big] \Big| \le [f]_\gamma c^n.
	\end{align}
	Moreover, there are constants $d>0$ and $D>0$ such that for all $t>0$ and all $f \in \mathcal{B}_\gamma$, 
	\begin{align}\label{eq:onepoint.theorem2}
		\Big| \EE \big[f(Z_{t}^y)\big] - \pi(f) \Big| \le D [f]_\gamma e^{-dt}.
	\end{align}
\end{proposition}

\begin{proof}
	We start by proving \eqref{eq:onepoint.theorem1}.
	We will use the following estimate for the angular distance between $Z_n^y$ and $Z_n^z$, for different starting points $y$ and $z$. By \cite[Cor. 3.18]{Guivarch2016} (applied with $s=0$) there is $\gamma>0$, $0<c<1$ and $n_0 \in \N$, such that for all $n\ge n_0$ 
	\begin{align}\label{eq:Hoelder1}
		\sup_{u,w} \E \bigg[ \frac{\mathbf{d}^\gamma(u\cdot X_n,w\cdot X_n)}{\mathbf{d}^\gamma(u,w)} \bigg] < c^n.
	\end{align}
	
	Let $0 \le r \le 1$. Using stationarity of the increments of $X$, Hölder continuity and the fact that the angular distance $\mathbf{d}$ is bounded by one, we obtain for any $y,z \in P(\R^d)$
	\begin{align*}
		\Big|\EE \big[f(Z_{n+r}^y)\big] - \EE \big[f(Z_n^z) \big] \Big|
		= & \Big| \EE \Big[f\big(Z_r^y \cdot X_{r,n+r} \big)\Big] - \EE \big[f(z \cdot X_{r,n+r})\big] \Big| \\
		\le & [f]_\gamma  \E \bigg[\mathbf{d}^\gamma(Z_r^y,z) \E \bigg[ \frac{\mathbf{d}^\gamma(Z_r^y\cdot X_{r,n+r},z\cdot X_{r,n+r})}{\mathbf{d}^\gamma(Z_r^y,z)}  \, \bigg| \mathcal{F}_r \bigg] \, \bigg] \\
		\le & [f]_\gamma \sup_{u,w} \E \bigg[ \frac{\mathbf{d}^\gamma(u\cdot X_n,w\cdot X_n)}{\mathbf{d}^\gamma(u,w)} \bigg].
	\end{align*}
	Observing that the upper bound is independent of both $r$ and $y$, the assertion \eqref{eq:onepoint.theorem1} follows by employing \eqref{eq:Hoelder1}.
	
	To obtain \eqref{eq:onepoint.theorem2}, we write $t=n+r$ for $r \in [0,1)$ and obtain, using \eqref{eq:onepoint.theorem1}
	\begin{align*}
		\abs{\EE \big[f(Z_{t}^y)\big] - \pi(f)} &=  \bigg|\EE \big[f(Z_{n+r}^y)\big] - \int_{P(\R^d)}\EE \big[f(Z_n^z) \big]  \pi(dz)\bigg| \\
		&= \bigg|\int_{P(\R^d)} \Big(\EE \big[f(Z_{n+r}^y)\big] - \EE \big[f(Z_n^z) \big]\Big)   \pi(dz)\bigg| \\
		&\le  \int_{P(\R^d)}\bigg| \Big(\EE \big[f(Z_{n+r}^y)\big] - \EE \big[f(Z_n^z) \big]\Big)  \bigg| \pi(dz) \\
		&\le [f]_\gamma c^n \le D [f]_\gamma e^{-dt}
	\end{align*}
	with $d= - \log c$ and $D = c^{-1}$.
	\end{proof}

\begin{example}\label{ex:invariant measure}
	Suppose that for any $h>0$, the law of $L$ is invariant under $\operatorname{O}(\RR,d)$, i.e. $\mathcal{L}(L_h)=\mathcal{L}(k^{-1}L_hk)$ for any $k \in \operatorname{O}(\RR,d)$. This property holds e.g. if $\Sigma_L=\Id$, $\gamma_L=0$ and $\nu_L=0$. 

	Then the law of $\Id+L_h$ is invariant under rotations as well,
	since for any measurable set $A \subset \R^{d \times d}$
	\begin{align*}
		\PP(k^{-1}(\Id + L_h)k \in A) &= \PP( \Id  + k^{-1}L_hk \in A) = \PP(k^{-1}L_hk \in (A-\Id)) = \PP(L_h \in (A-\Id)) \\
		 &= \PP(\Id + L_h \in A).
	\end{align*}
The unique probability measure $\xi$ on $P(\R^d)$ that is invariant under the action of $\Od$ is defined by
	$$ \int_{P(\R^d)} f(x) \xi(dx) = \int_{\Od} f(x_0k) dk, \quad \text{ for any } f \in \mathcal{C}(P(\R^d)),$$
	where $dk$ denotes the Haar measure on $\Od$, normalised to have total mass 1, and $x_0$ can be an arbitrary, but fixed element of $P(\R^d)$.
	Then 
	\begin{align*}
	\lefteqn{\EE \Big[ \int_{P(\R^d)} f\Big(x\cdot(\Id+L_h)\Big) \, \xi(dx) \Big]
		= \EE \Big[ \int_{\Od} f\Big(x_0k \cdot (\Id+L_h)\Big) \, dk \Big] }\\
		&=   \EE \Big[ \int_{\Od} f\Big(x_0k \cdot k^{-1}(\Id+L_h)k\Big) \, dk \Big] = \EE \Big[ \int_{\Od} f\Big(\big(x_0\cdot (\Id+L_h)\big)k\Big) \, dk \Big] \\
		& =  \EE \Big[ \int_{P(\R^d)} f (x) \xi(dx) \Big] = \int_{P(\R^d)} f (x) \xi(dx).
	\end{align*}
	Hence $\xi$ is invariant for the action of $\Id + L_h$ for any $h$. Using the approximation given by Lemma \ref{lemma-emeryapprox} with a sequence $\sigma_\ell=(j 2^{-\ell}; j \in \N)$, we obtain that $\xi$ is invariant for the action of $X_1^{\sigma_\ell}$ for every $\ell$, and hence, using dominated convergence
	\begin{align*} 
		\lefteqn{\EE \Big[ \int_{P(\R^d)} f(x\cdot X_1) \xi(dx) \Big] = \EE \Big[ \int_{P(\R^d)} \lim_{\ell \to \infty} f(x\cdot X_1^{\sigma_\ell}) \xi(dx) \Big] } \\
		& = \lim_{\ell \to \infty} \EE \Big[ \int_{P(\R^d)}  f(x\cdot X_1^{\sigma_\ell}) \xi(dx) \Big] = \int_{P(\R^d)}  f(x) \xi(dx),
	\end{align*}
	we see that $\xi$ is invariant for $Z_t$. In other words, $\pi=\xi$.
\end{example}

\begin{example}\label{ex:invariant measure2}
	Suppose that $G_L$ is concentrated on the subset of non-negative, invertible matrices. Then the positive orthant, $C=[0,\infty)^d$, is invariant under $G_L$. 
	Since $\Id+G_L$ is again concentrated on the subset of non-negative matrices, it holds that $C=[0,\infty)^d$ is invariant under $\Id+L_h$ as well, for any $h>0$. 
	
	Then  the action of $\Id+L_h$ has an invariant set $\hat{C}$, namely the projection of $C$ on $P(\R^d)$. Hence, the invariant measure $\pi$ will also be concentrated on $\hat{C}$.
\end{example}

\subsection{The strong law of large numbers and the central limit theorem}

We are now ready to prove the SLLN and CLT for the logarithm of norm and entries of $X_t$, by relying on corresponding results for the discrete time skeleton. The first two lemmas show that under appropriate moment assumptions, we may neglect small-time contributions.

\begin{theorem}\label{CLT} Let $L$ be a L\'evy process in $\RR^{d\times d}$ fulfilling \eqref{eq-nonsingular} and let $X$ be the corresponding stochastic exponential.
	Assume that $G_X$  is unbounded and acts strongly irreducibly on $\R^d$, and that $\E \big[\sup_{0 \le s \le 1} (\log M(X_s))^2\big]<\infty$.
	
	Then there is $\lambda \in \R$ such that
	\begin{equation}\label{eq:SLLN:theorem}
		 \frac{1}{t} \log F(X_t) \to \lambda \quad \as,
	\end{equation}
	and $\sigma^2>0$ such that
	\begin{equation}\label{eq:CLT:theorem}
		\frac{\log F(X_t) - t \lambda}{\sqrt{t}} \stackrel{d}{\to} \mathcal{N}(0, \sigma^2), \quad t\to \infty,
	\end{equation}
	where $F : \GL \to [0,\infty)$ can be chosen as $F(a)=\norm{a}$ as well as $F(a)=\norm{ya}$ for any $y \in \R^d$ with $\norm{y}=1$.
\end{theorem}

\begin{remark}
	There exist several expressions for $\lambda$ and $\sigma^2$ in the literature, however, all of them involve either computing a limit or evaluating an integral with respect to the stationary measure of $Z_t^x$.
	\begin{itemize}
			\item If $\E \big[ M(X_1)^\delta\big]<\infty$ for some $\delta>0$, then the limit
		$$ \Lambda(s):=\lim_{n \to \infty} \frac1n \log \EE \big[ \norm{X_n}^s\big]$$ is well defined for $s \in [0,\delta)$, see e.g. \cite[Sect. 4.4.3]{Buraczewski2016a}, and has an analytic extension  on $(-\delta,\delta)$, see e.g. \cite[Prop. 3.1]{Xiao2022} in connection with  \cite[Thm. 2.6]{Guivarch2016}. Then  $\lambda=\Lambda'(0)$ and $\sigma^2=\Lambda''(0)$ by \cite[Props. 3.13 and 3.15]{Xiao2022}.  In a sense, $\Lambda$ can be understood as an analogue of a cumulant generating function.
		\item The convergence in \eqref{eq:SLLN:theorem} holds also in $L^1$ (see \cite{Furstenberg1960}), hence $\lambda=\lim_{n \to \infty} \frac{1}{n} \EE \big[ \log \norm{X_n}\big]$. In addition, if $\E \big[ M(X_1)^\delta\big]<\infty$ for some $\delta>0$, then by \cite[Prop. 14.17]{Benoist2016a},
		$$ \sigma^2= \lim_{n \to \infty} \frac{1}{n} \EE \Big[ \big(\log \norm{X_n} - n \lambda\big)^2 \Big].$$
		\item By \cite[Prop. III.7.2]{Bougerol1985}, it further holds  
		$$\lambda=\int \EE \log \abs{Y_1^x} \pi(dx).$$ 
	\end{itemize}
\end{remark}

\begin{proof}[Proof of Theorem \ref{CLT}]
	Recalling that $X_n=\prod_{k=1}^n X_{k-1,k}$ is a product of i.i.d.\ random invertible matrices, we may
    apply \cite[Thm. 2]{Furstenberg1960} and \cite[Thm. 1.1]{Benoist2016} which provide us with the validity of \eqref{eq:SLLN:theorem} and \eqref{eq:CLT:theorem}, respectively, along the discrete skeleton $(X_n)_{n \in \N}$. 
    
    To prove the limit theorems with a continuous time parameter, consider any sequence $(t_k)_{k \in \N}$ with $t_k \to \infty$ as $k \to \infty$. Then we may choose a corresponding sequence of integers $n_k$ such that $n_k \le t_k < n_k+1$, hence it holds as well that $n_k \to \infty$ as $k \to \infty$.
    
    Turning first to the strong law of large numbers \eqref{eq:SLLN:theorem} with $F(\cdot)=\norm{\cdot}$, we have for any $t>0$ and $n \in \N$ such that $n \le t < n+1$
    \begin{align*}
    	\frac{1}{t} \log \norm{X_{t}} = \frac{n}{t} \cdot  \frac{1}{n}\Big( \log \norm{X_{t}} - \log\norm{X_{n}} +  \log\norm{X_{n}}  \Big).
    \end{align*}
    Since $n/t\to 1$  and, by the result for the discrete-time skeleton, $n^{-1} \log \norm{X_{n}} \to \lambda$ \as\ for $n \to \infty$, the result follows once we have proved 
    that $\sup_{0 \le s \le 1} \big| \log \norm{X_{n+s}} - \log \norm{X_n} \big|  $
     vanishes \as\ for $n \to \infty$.
    
    For an invertible matrix $g$, we write $\overline{g}:= \norm{g}^{-1} g$. Using \eqref{eq:bound for log norm ga}, we obtain
    \begin{align*}
    	 \lefteqn{\sup_{0 \le s \le 1} \big| \log \norm{X_{n+s}} - \log \norm{X_n} \big|  
    	 = \sup_{0 \le s \le 1} \big| \log \norm{X_n X_{n,n+s}} - \log \norm{X_n} \big| }\\
    	& = \sup_{0 \le s \le 1} \Big| \log \frac{\norm{X_n X_{n,n+s}}}{\norm{X_n}} \Big| = \sup_{0 \le s \le 1} \big| \log \norm{\overline{X_n} X_{n,n+s}} \big| \le \sup_{0 \le s \le 1} \log M(X_{n,n+s}) .
    \end{align*}
    Invoking Lemma \ref{lem:M(Xs).vanishes} (with $\kappa=1$), we hence conclude that
    $$
    	\lim_{n \to \infty} \sup_{0 \le s \le 1} \big| \log \norm{X_{n+s}} - \log \norm{X_n} \big| = 0 \quad \as 
    $$
 
 Turning now to the central limit theorem \eqref{eq:CLT:theorem}, we decompose for $n \le t < n+1$
 \begin{align*}
 	\frac{\log \norm{X_t} - t \lambda}{\sqrt{t}} =& \frac{\sqrt{n}}{\sqrt{t}} \frac{1}{\sqrt{n}} \Big( \log \norm{X_t} - \log \norm{X_n} + \log \norm{X_n} - n \lambda + n \lambda - t\lambda  \Big) \\
 	= & \frac{\sqrt{n}}{\sqrt{t}} \frac{\log \norm{X_n} - n \lambda}{\sqrt{n}} + \frac{\sqrt{n}}{\sqrt{t}} \frac{1}{n^{1/2}} \Big( \log \norm{X_t} - \log \norm{X_n})\Big) - \frac{(t-n)\lambda}{\sqrt{t}} . 
 \end{align*}
 
 By the same arguments as before, now using Lemma \ref{lem:M(Xs).vanishes} with $\kappa=2$, we deduce that the middle term vanishes \as; as does the last term. The first term converges in distribution by the central limit theorem for the discrete-time skeleton. By Slutsky's theorem, we conclude the convergence in distribution of the full term to $\mathcal{N}(0, \sigma^2)$. Here, $\sigma^2$ is the same as for the discrete-time sceleton $(X_n)_{n \in \N}$. 
 
 To obtain the strong law of large numbers and the central limit theorem for $F(a)=\log \norm{ya}$, we use \cite[Thm. 4.11]{Benoist2016} which provides us with the corresponding results for the discrete-time skeleton. Then the proof goes along the same lines, using that \eqref{eq:bound for log norm ga} holds for vectors as well.
\end{proof}

\begin{theorem}\label{CLTcoefficients}
	In addition to the assumptions of Theorem~\ref{CLT}, assume that $G_X$ is proximal and that $\EE \big[\sup_{0 \le s \le 1} M(X_s)^\epsilon\big]<\infty$ for some $\epsilon>0$. Then 
	the CLT \eqref{eq:CLT:theorem} holds for $F(a)=\abs{\skalar{ya,z}}$ as well, for any $y,z \in \R^d$ with $\norm{y}=\norm{z}=1$.\\
	In particular, choosing $y=e_i$ and $z=e_j$ as unit vectors, there exist $\lambda\in\RR$, $\sigma^2>0$ such that 
			\begin{equation*}
			\frac{\log X_t^{(i,j)} - t \lambda}{\sqrt{t}} \stackrel{d}{\to} \mathcal{N}(0, \sigma^2), \quad t\to \infty.
		\end{equation*}
\end{theorem}

\begin{proof}
	We follow the strategy of \cite[Sec. 14.4]{Benoist2016a}, i.e., we compare $\log \norm{yX_t}$ with $\log\abs{\skalar{yX_t,z}}$. Introducing the quantity
	$$ \delta(y,z) := \frac{\abs{\skalar{y,z}}}{\norm{y}\norm{z}},$$
	we note $\delta \le 1$ and observe that 
	\begin{align}\label{eq:property delta}
		\log\abs{\skalar{yX_t,z}}=\log \norm{yX_t} + \log \delta \big(yX_t,z\big).
	\end{align}
	 
	 By \cite[Prop. 14.3]{Benoist2016a}, for any $\epsilon>0$, there is $c>0$ and $n_0 \in \N$ such that for all $n \ge n_0$, $y,z \in \R^d$ with $\norm{y}=\norm{z}=1$ one has
	 $$ \PP \big( \log \delta(yX_n,z) \le -\epsilon n \big) \le e^{-c n}.$$
	 
	 Consider now any sequence $(t_k)_{k \in \N}$ with $t_k \to \infty$ as $k \to \infty$, let $(n_k)_{k \in \N}$ be an integer sequence with $n_k \le t_k < n_k+1$ for all $k$ and set $s_k=t_k-n_k$. Then	 
	 \begin{align*}
	 	\PP \big( \log \delta(yX_{t_k},z) \le -\epsilon n_k \big) &= \int  \PP \big( \log \delta((u X_{s_k, n_k+s_k},z) \le -\epsilon n_k \big) \, \PP (yX_{s_k} \in du) \\
	 	&\le e^{-c n_k}. 
	 \end{align*}
	 Using that $\epsilon>0$ was arbitrary, $n_k/t_k \to 1$, and $\log \delta(\cdot)\le 0$, we conclude 
	 $$ \PP\text{-}\lim_{t \to \infty} \frac{1}{t} \log \delta(yX_{t_k},z)= 0.$$
	 Recalling \eqref{eq:property delta}, the asserted CLT then follows by an application of Slutsky's theorem from the corresponding result for $\log \norm{yX_t}$, Theorem~\ref{CLT}.
	\end{proof}

\subsection{Berry-Esseen bounds and convergence of the one-point motion}

In this section we extend the central limit theorem presented in the last subsection by providing Berry-Esseen bounds as well as  joint convergence of the directional and radial components, $Z_t^y$ and $\log \norm{Y_t^y}$, respectively, of the one-point motion. We rely on results for (discrete-time) products of random matrices in \cite{Xiao2022} and \cite{Cuny2023}.

\begin{theorem}\label{BerryEsseen} Let $L$ be a L\'evy process in $\RR^{d\times d}$ fulfilling \eqref{eq-nonsingular} and let $X$ be the corresponding stochastic exponential.
	Assume that $G_X$ satisfies  (i-p) and $\EE \big[ \sup_{0 \le s \le 1} (\log M(X_s))^4 \big]<\infty$. Then there is a constant $C>0$ such that for all $n \ge 1$, $y \in \R^d$ with $\norm{y}=1$ and $z \in \R$,
\begin{equation}\label{eq:BerryEsseenTheorem0}
	\bigg| \PP \bigg(\frac{\log \norm{Y_t^y}-t\lambda}{\sigma \sqrt{t}} \le z\bigg) -  \Phi(z)  \bigg| ~\le~\frac{C}{\sqrt{t}}. 
\end{equation}
Under the additional assumption that $\EE \big[ M(X_1)^\delta\big]<\infty$ for some $\delta>0$,  there exist $C>0$ and $\gamma>0$ such that for all $n \ge 1$, $y \in \R^d$ with $\norm{y}=1$, $z \in \R$ and $\varphi \in \mathcal{B}_\gamma$,
	\begin{equation}\label{eq:BerryEsseenTheorem2}
		\bigg| \EE \bigg[ \varphi(Z_t^{\hat{y}}) \bone\bigg\{\frac{\log \norm{Y_t^y}-t\lambda}{\sigma \sqrt{t}} \le z\bigg\}\bigg] - \pi(\varphi) \Phi(z)  \bigg| ~\le~\frac{C}{\sqrt{t}} \norm{\varphi}_\gamma .
	\end{equation}
\end{theorem}

\begin{proof}
	For discrete $t \in \N$, the Berry-Esseen result under the fourth moment assumption, \eqref{eq:BerryEsseenTheorem0}, is proved in \cite[Thm. 2.2]{Cuny2023}. Under the additional moment assumption, the Berry-Esseen bound \eqref{eq:BerryEsseenTheorem2} for the joint convergence of $Z_t^u$ and $\log \norm{Y_t^u}$ is \cite[Thm. 2.1]{Xiao2022} in the discrete setting $t \in \N$. We show how to deduce \eqref{eq:BerryEsseenTheorem2} from the corresponding discrete-time result, the proof for \eqref{eq:BerryEsseenTheorem0} being exactly the same by setting $\varphi \equiv 1$. Note that in our estimations, we will only need finite moments of $\log M(X_s)$.
	
	Let $t=n+r$ for $n \in \N$, $0 \le r <1$. Applying \eqref{eq:BerryEsseenTheorem2}, that is \cite[Thm. 2.1]{Xiao2022}, for the discrete sceleton with $y=Z_r^u$, we have
	\begin{align*}
	\lefteqn{ \Big| \EE \bigg[ \varphi(Z_t^{\hat{u}}) \bone\bigg\{\frac{\log \norm{Y_t^u}-t\lambda}{\sigma \sqrt{t}} \le z\bigg\}\bigg] - \pi(\varphi) \Phi(z)  \Big|} \\
		& \le  \E \bigg[ \Big| \EE \bigg[ \varphi(Z_r^{\hat{u}} \cdot X_{r,n+r}) \bone\bigg\{\frac{\log \norm{Y_r^u X_{r,n+r}}-t\lambda}{\sigma \sqrt{t}} \le z\bigg\}\bigg] - \pi(\varphi) \Phi(z)  \Big| \, \bigg| \, \mathcal{F}_r \bigg] \\
		& =  \E \bigg[ \Big| \EE \bigg[ \varphi(Z_r^{\hat{u}} \cdot X_{r,n+r}) \bone\bigg\{\frac{\log \norm{Z_r^{\hat u} X_{r,n+r}}-n\lambda}{\sigma \sqrt{n}} \le \frac{\sqrt{t}}{\sqrt{n}} \big( z- \frac{\log \norm{Y_r^u}}{\sigma \sqrt{t}}+\frac{r \lambda}{\sigma \sqrt{t}}\big)\bigg\}\bigg] - \pi(\varphi) \Phi(z)  \Big| \, \bigg| \, \mathcal{F}_r \bigg] \\
		& \le  \frac{C}{\sqrt{n}} \norm{\varphi}_\gamma + \abs{\pi(\varphi)} \E \bigg[ \Big|\Phi\Big(\frac{\sqrt{t}}{\sqrt{n}}  z- \frac{\log \norm{Y_r^u}}{\sigma \sqrt{n}}+\frac{r \lambda}{\sigma \sqrt{n}}\Big) - \Phi(z) \Big| \bigg],
	\end{align*}
since $ \frac{Y^u_r }{\norm{Y^u_r}}=Z^{\hat{u}}_r.$
To estimate the last term, we use $\abs{\Phi(x)-\Phi(y)} =\abs{\int_{[y,x]} \Phi'(z) dz} \le \abs{x-y}$ and obtain
\begin{align}
	\lefteqn{\E \bigg[ \Big|\Phi\Big(\frac{\sqrt{t}}{\sqrt{n}}  z- \frac{\log \norm{Y_r^u}}{\sigma \sqrt{n}}+\frac{r \lambda}{\sigma \sqrt{n}}\Big) - \Phi(\tfrac{\sqrt{t}}{\sqrt{n}}  z) + \Big( \Phi(\tfrac{\sqrt{t}}{\sqrt{n}}  z) - \Phi(z)\Big) \Big| \bigg] } \nonumber \\
	&	\le  \frac{1}{\sigma \sqrt{n}}\E \Big[\abs{ \log \norm{Y_r^u}} \Big] + \frac{r\lambda}{\sigma \sqrt{n}} + \Big( \Phi(\tfrac{\sqrt{t}}{\sqrt{n}}  z) - \Phi(z)\Big). \label{eq:phi}
\end{align}
	Using Lemma \ref{lem:EMa}, we have that for any $u \in \R^d$ with $\norm{u}=1$ and for any $r \in [0,1]$, $\E \Big[\abs{ \log \norm{Y_r^u}} \Big] \le \E \Big[ \sup_{0\le s \le1} \log M(X_s)\Big]$, and the latter is finite due to our moment assumption. 
	To estimate the last term in \eqref{eq:phi}, we assume w.l.o.g. $z>0$ (otherwise, use the symmetry of $\Phi$.). By the mean value theorem there is $z \le \xi \le cz$ such that
	\begin{align*}
		\abs{\Phi(\tfrac{\sqrt{t}}{\sqrt{n}}z)-\Phi(z)} \le z\Phi'(\xi) (\tfrac{\sqrt{t}}{\sqrt{n}}-1) \le z \Phi'(z) \frac{\sqrt{n+r}-\sqrt{n}}{\sqrt{n}} \le \frac{\sqrt{r}}{\sqrt{n}}.
	\end{align*}
	Here we have also used that $z\Phi'(z) \le z e^{-\tfrac{1}{2} z^2} \le 1$ for all $z$, and the subadditivity $\sqrt{n+r} \le \sqrt{n}+\sqrt{r}$.
	
	We thus have proved
	\begin{align*}
		 \bigg| \EE \Big[ \varphi(Z_t^{\hat{u}}) \bone_{\left\{\frac{\log \norm{Y_t^u}-t\lambda}{\sigma \sqrt{t}} \le z \right\}}\Big] - \pi(\varphi) \Phi(z)  \bigg| 	\le  \frac{C}{\sqrt{n}} \norm{\varphi}_\gamma + \abs{\pi(\varphi)} \frac{C' + r\lambda + \sigma\sqrt{r}}{\sigma\sqrt{n}},
	\end{align*}
	and the assertion follows as $\sqrt{n}/\sqrt{t} \to 1$, $r \le 1$, and $\abs{\pi(\varphi)} \le \norm{\varphi}_\infty \le \norm{\varphi}_\gamma$.
\end{proof}

\section{Limit theorems for the determinant process}\label{sect:determinant}

In this section we derive asymptotics for the determinant of the stochastic exponential. As we will see in the next subsection, the determinant process has a particularly nice structure. In comparison to Section \ref{sect:limitNorms} this allows for completely different techniques to be used in order to describe its long-time behaviour as we will do in Section \ref{sectdetLLN}. 

\subsection{Structure of the determinant process}

Our first result proves that the determinant process is again a multiplicative Lévy process.

\begin{theorem} \label{thm-detprocess}
	Let $L=(L_t)_{t\geq 0}$ be a Lévy process in $\RR^{d\times d}$ with characteristic triplet $(\Sigma_L, \gamma_L, \nu_L)$ and fulfilling \eqref{eq-nonsingular} and let $D=(D_t)_{t\geq 0}$ be the determinant process of its stochastic exponential defined in \eqref{eq-detprocess}. 
	Then $D$ is a left (and right) Lévy process taking values in $\RR\setminus\{0\}$. In particular, $D_t=\cE(\check{L})_t$ for the real-valued (additive) Lévy process $\check{L}=(\check{L}_t)_{t\geq 0}$ given by
	$$\check{L}_t= \trace(L_t) + \frac12 \sum_{\substack{m,n=1\\m\neq n}}^d (\sigma_{(m,m),(n,n)} - \sigma_{(m,n),(n,m)})\cdot  t  + \sum_{0<s\leq t} (\det(\Id + \Delta L_s)-1- \trace(\Delta L_s)), \quad t\geq 0,$$
where $\sigma_{(m,j),(n,\ell)}$ denotes the covariance between the Gaussian components of $L^{(m,j)}$ and $L^{(n,\ell)}$. 
\end{theorem}
\begin{proof}
	The fact that $D$ is a left Lévy process can be easily shown by checking the defining properties: Clearly $D_0=\det X_0= \det \Id = 1$ a.s. Moreover, stationarity of the increments follows from the same property of $X$ as for all $s\leq t$
	\begin{align*}
		D_{s,t} &= D_s^{-1} D_t = (\det X_s)^{-1} \det (X_t) = \det(X_s^{-1} X_t)= \det (X_{s,t}) \\
		&\overset{d}= \det (X_{0, t-s}) = \ldots = D_{0,t-s}.
	\end{align*}
	Further, as $D_{s,t}= \det (X_{s,t})$ for all $s\leq t$, independence of the increments is immediate. Lastly, as $L$ has c\`adl\`ag paths by definition, the same is true for $X$ as the solution of \eqref{SDE} and hence for $D$ since the determinant mapping is continuous.\\
	To derive the given formula for $\check{L}$ we use a multivariate version of It\^o's formula as stated e.g. in \cite[Thm. II.33]{protter} which reads
	\begin{align}
		\lefteqn{f(X_t)-f(X_0)} \nonumber \\ & = \sum_{i,j=1}^d \int_{(0,t]} \frac{\partial f}{\partial x^{(i,j)}}(X_{s-}) d X_s^{(i,j)} + \frac12 \sum_{i,j,k,\ell=1}^d \int_{(0,t]} \frac{\partial^2 f}{\partial x^{(i,j)} \partial x^{(k,\ell)}} (X_{s-}) d[X^{(i,j)},X^{(k,\ell)}]_s^c \nonumber \\
		&\quad +\sum_{0<s\leq t} \left(f(X_s)-f(X_{s-}) - \sum_{i,j=1}^d \frac{\partial f}{\partial x^{(i,j)}} (X_{s-}) \Delta X_s^{(i,j)} \right) \label{eq-itodet}
	\end{align}
	for $f:\RR^{d\times d}\to \RR$ twice continuously differentiable. Choosing $f(X)=\det X$ we can now evaluate the above formula term-by-term as follows.\\
	From \cite[Eqs. (10.7.2) and (10.7.11)]{Bernstein} we obtain
	\begin{equation}\label{eq-detderivative}	\frac{\partial f}{\partial x^{(i,j)}}(X_{s-}) = (X_{s-}^{\adj})^{(j,i)} = (-1)^{i+j} \det(X_{s-,[i;j]}),\end{equation}
	where the superscript $\adj$ denotes the adjugate of a matrix, while the subscript $[i;j]$ denotes deletion of the $i$-th row and $j$-th column. With this we obtain for the first integral term in \eqref{eq-itodet} in similarity to the classical Jacobi equality (cf. \cite[Fact 10.12.8]{Bernstein})
	\begin{align*}
		\sum_{i,j=1}^d \int_{(0,t]} \frac{\partial f}{\partial x^{(i,j)}}(X_{s-}) d X_s^{(i,j)}&= \sum_{i,j=1}^d \int_{(0,t]} (X_{s-}^{\adj})^{(j,i)}  d X_s^{(i,j)} = \sum_{j=1}^d \left(\int_{(0,t]} (X_{s-}^{\adj})  d X_s \right)^{(j,j)}\\
		&= \trace\left( \int_{(0,t]}(X_{s-}^{\adj})  d X_s \right).
	\end{align*}
	Inserting \eqref{SDE} and using  \cite[Eq. (2.7.21)]{Bernstein} the above further equals
	\begin{align}
		\trace\left( \int_{(0,t]}(X_{s-}^{\adj}) X_{s-} d L_s \right)&= \trace\left(\int_{(0,t]} (\det X_{s-}) \Id d L_s \right) = \sum_{j=1}^d \left( \int_{(0,t]} D_{s-} \Id d L_s \right)^{(j,j)} \nonumber \\
		&= \sum_{j=1}^d   \sum_{\ell=1}^d \int_{(0,t]} D_{s-} \Id^{(j,\ell)} d L_s^{(\ell,j)}  = \sum_{j=1}^d \int_{(0,t]} D_{s-}  d L_s^{(j,j)} \nonumber \\
		&= \int_{(0,t]} D_{s-} d \trace(L_s). \label{eq-itodetterm1}
	\end{align}
	Concerning the jump part in \eqref{eq-itodet} we obtain by similar means
	\begin{align*}
		\sum_{i,j=1}^d \frac{\partial f}{\partial x^{(i,j)}} (X_{s-}) \Delta X_s^{(i,j)} &= 	\sum_{i,j=1}^d (X_{s-}^{\adj} )^{(j,i)}  \Delta X_s^{(i,j)} = \sum_{i,j=1}^d (X_{s-}^{\adj})^{(j,i)} \left(\sum_{\ell=1}^d  X_{s-}^{(i,\ell)} \Delta L_s^{(\ell,j)} \right) \\
		&=  \sum_{j,\ell=1}^d (X_{s-}^{\adj} X_{s-})^{(j,\ell)}    \Delta L_s^{(\ell,j)} = \sum_{j,\ell=1}^d (D_{s-} \Id )^{(j,\ell)}    \Delta L_s^{(\ell,j)}\\
		&= D_{s-} \trace(\Delta L_s).
	\end{align*}
	Adding the fact that
	$$f(X_s)= \det(X_{s-}X_{s-,s}) = \det X_{s-} \det(X_{s-,s}) = D_{s-} \det(\Id + \Delta L_s)$$
	we thus get
	\begin{align}
		\lefteqn{\sum_{0<s\leq t} \left(f(X_s)-f(X_{s-}) - \sum_{i,j=1}^d \frac{\partial f}{\partial x^{(i,j)}} (X_{s-}) \Delta X_s^{(i,j)} \right)}\nonumber \\
		&= \sum_{0<s\leq t} \left(D_{s-} \det(\Id + \Delta L_s)  - D_{s-} - D_{s-} \trace(\Delta L_s) \right) \nonumber \\
		&= \sum_{0<s\leq t} D_{s-} \left(\det(\Id + \Delta L_s)  - 1 -  \trace(\Delta L_s) \right).\label{eq-itodetterm3}
	\end{align}
	Lastly, for the second integral term in \eqref{eq-itodet}, note first that
	\begin{align*}
		[X^{(i,j)},X^{(k,\ell)}]_s^c &= \left[ \Big(\int_{(0,\cdot]} X_{u-} dL_u\Big)^{(i,j)}, \Big(\int_{(0,\cdot]} X_{u-} dL_u\Big)^{(k,\ell)} \right]_s^c\\
		&= \left[ \sum_{m=1}^d \int_{(0,\cdot]} X_{u-}^{(i,m)} dL_u^{(m,j)}, \sum_{n=1}^d \int_{(0,\cdot]} X_{u-}^{(k,n)} dL_u^{(n,\ell)}  \right]_s^c\\
		&= \sum_{m,n=1}^d \left[ \int_{(0,\cdot]} X_{u-}^{(i,m)} dL_u^{(m,j)}, \int_{(0,\cdot]} X_{u-}^{(k,n)} dL_u^{(n,\ell)}  \right]_s^c\\
		&= \sum_{m,n=1}^d  \int_{(0,s]} X_{u-}^{(i,m)} X_{u-}^{(k,n)} d[L^{(m,j)}, L^{(n,\ell)}]_u^c\\
		&= \sum_{m,n=1}^d \sigma_{(m,j),(n,\ell)}  \int_{(0,s]} X_{u-}^{(i,m)} X_{u-}^{(k,n)} du.
	\end{align*}
	Further, iterating \eqref{eq-detderivative} via \cite[Eq. (10.7.11)]{Bernstein} yields
	\begin{align*}\lefteqn{\frac{\partial^2 f}{\partial x^{(i,j)} \partial x^{(k,\ell)}}(X_{s-})  =  (-1)^{i+j}	\frac{\partial f}{\partial x^{(k,\ell)}} \det(X_{s-,[i;j]} )}\\
		& = \begin{cases}
			0, & i=k,\,\text{or}\, j=\ell,\\
			(-1)^{i+j} ((X_{s-,[i;j]})^{\adj})^{(\ell',k')}  = (-1)^{i+j+k'+\ell'} \det(X_{s-,[i,k;j,\ell]}), & \text{else},
	\end{cases} \end{align*} 
	where the subscript $[i,k;j,\ell]$ denotes deletion of the $i$-th and $k$-th row as well as the $j$-th and $\ell$-th column. Moreover, we set $k'= k$ for $k<i$, and $k'=k-1$ for $k>i$, and define $\ell'$ analogously w.r.t. $j$. Thus
	\begin{align}
		\lefteqn{\sum_{i,j,k,\ell=1}^d \int_{(0,t]} \frac{\partial^2 f}{\partial x^{(i,j)} \partial x^{(k,\ell)}} (X_{s-}) d[X^{(i,j)},X^{(k,\ell)}]_s^c }\nonumber \\
		&= \sum_{\substack{i,j,k,\ell,m,n =1\\ i\neq k, j \neq \ell}}^d (-1)^{i+j+k'+\ell'}  \sigma_{(m,j),(n,\ell)} \int_{(0,t]} \det(X_{s-,[i,k;j,\ell]}) X_{s-}^{(i,m)} X_{s-}^{(k,n)} ds \nonumber \\
		&= \sum_{\substack{j,\ell,m,n =1\\ j \neq \ell}}^d  \sigma_{(m,j),(n,\ell)}   \int_{(0,t]}  (-1)^{j +\ell'}  \sum_{\substack{i,k=1\\i\neq k}}^d (-1)^{i+k'} \det(X_{s,[i,k;j,\ell]}) X_{s}^{(i,m)} X_{s}^{(k,n)} ds. \label{eq-itodetterm2help1}
	\end{align}
	To simplify this sum, first of all observe that $\det(X_{s,[i,k;j,\ell]})= \det(X_{s,[k,i;j,\ell]})$ such that we can rewrite the inner sum in \eqref{eq-itodetterm2help1} as
	\begin{align*}
		\lefteqn{\sum_{\substack{i,k=1\\i\neq k}}^d (-1)^{i+k'} \det(X_{s,[i,k;j,\ell]}) X_{s}^{(i,m)} X_{s}^{(k,n)}}\\
		&= \sum_{i=1}^{d-1} \sum_{k=i+1}^d (-1)^{i+k-1} \det(X_{s,[i,k;j,\ell]})\underbrace{\left( X_s^{(i,m)} X_s^{(k,n)} - X_s^{(k,m)} X_s^{(i,n)} \right)}_{=0, \text{ if } m=n}.
	\end{align*}
	Thus, excluding all terms with $m=n$, we further note that by double Laplace expansion along the $m$-th and $n$-th column, again using $\det(X_{s,[i,k;j,\ell]})= \det(X_{s,[k,i;j,\ell]})$,
	\begin{align*}
		\sum_{\substack{i,k=1\\i\neq k}}^d (-1)^{i+k'+j+\ell'} \det(X_{s,[i,k;j,\ell]}) X_{s}^{(i,m)} X_{s}^{(k,n)} &= \begin{cases} 
			\det(X_s),  & j=m, \ell=n\\
			- \det(X_s), & j=n, \ell=m.
		\end{cases}
	\end{align*}
	Therefore, \eqref{eq-itodetterm2help1} equals
	\begin{align}
		\lefteqn{\sum_{\substack{j,\ell,m,n =1\\ j \neq \ell, m\neq n}}^d  \sigma_{(m,j),(n,\ell)}   \int_{(0,t]}  (-1)^{j +\ell'}  \sum_{\substack{i,k=1\\i\neq k}}^d (-1)^{i+k'} \det(X_{s,[i,k;j,\ell]}) X_{s}^{(i,m)} X_{s}^{(k,n)} ds} \nonumber \\
		&= \bigg(\sum_{\substack{m,n=1\\m\neq n}}^d \sigma_{(m,m),(n,n)}  - \sum_{\substack{m,n=1\\m\neq n}}^d \sigma_{(m,n),(n,m)} \bigg) \int_{(0,t]} D_s ds \label{eq-itodetterm2help2}  \\
		&\qquad+ \sum_{\substack{j,\ell,m,n =1\\ j \neq \ell, m\neq n\\
				\lnot[(j=m) \wedge (\ell=n)]\\ \lnot[(j=n)\wedge (\ell=m)]}}^d  \sigma_{(m,j),(n,\ell)} (-1)^{j +\ell'}   \int_{(0,t]}   \sum_{\substack{i,k=1\\i\neq k}}^d (-1)^{i+k'} \det(X_{s,[i,k;j,\ell]}) X_{s}^{(i,m)} X_{s}^{(k,n)} ds. \label{eq-itodetterm2help3}
	\end{align}
	The last term, i.e. \eqref{eq-itodetterm2help3}, vanishes. To see this consider first the case that $m=j$, thus implying $n\neq \ell$. The resulting summands in \eqref{eq-itodetterm2help3} are of the form
	$$	\sigma_{(m,m),(n,\ell)} (-1)^{m +\ell'}   \int_{(0,t]}   \sum_{\substack{i,k=1\\i\neq k}}^d (-1)^{i+k'} \det(X_{s,[i,k;m,\ell]}) X_{s}^{(i,m)} X_{s}^{(k,n)} ds$$
	with
	\begin{align*}
		\lefteqn{\sum_{\substack{i,k=1\\i\neq k}}^d (-1)^{i+k'} \det(X_{s,[i,k;m,\ell]}) X_{s}^{(i,m)} X_{s}^{(k,n)}} \\
		&	 = \sum_{i=1}^d (-1)^i X_s^{(i,m)} \sum_{k'=1}^{d-1} (-1)^{k'} \det\big( (X_{s,[i;m]})_{[k';\ell']}\big) (X_{s,[i;m]})^{(k', n')}
		= 0,
	\end{align*}
	due to \cite[Eq. (2.7.17)]{Bernstein} applied on the transpose of $X_{s,[i;m]}$. By symmetry and recalling that $\det(X_{s,[i,k;j,\ell]})= \det(X_{s,[k,i;j,\ell]})$, all terms corresponding to the cases $n=\ell$, or $n=j$, or $\ell=m$ cancel out likewise due to \cite[Eq. (2.7.17)]{Bernstein}. Therefore only terms with all indices $m,j,n,\ell$ pairwise different remain to be considered, and clearly these only appear for $d\geq 4$. In this situation, w.l.o.g. assuming $m<j<\ell$,  we obtain
	\begin{align*}
		\lefteqn{\sigma_{(m,j),(n,\ell)} (-1)^{j +\ell'}  \sum_{\substack{i,k=1\\i\neq k}}^d (-1)^{i+k'} \det(X_{s,[i,k;j,\ell]}) X_{s}^{(i,m)} X_{s}^{(k,n)} }\\
		&= \sigma_{(m,j),(n,\ell)} (-1)^{j +\ell}  \sum_{k=1}^d  (-1)^{k} X_{s}^{(k,n)} \sum_{i'=1}^{d-1} (-1)^{i'} \det(X_{s,[k;\ell]})_{[i';j]}) (X_{s,[k;\ell]})^{(i',m)},
	\end{align*}
	and again by \cite[Eq. (2.7.17)]{Bernstein} the inner sum, and hence the whole term, equals zero. \\ 
	Finally, according to \eqref{eq-itodet}, we add up \eqref{eq-itodetterm1}, $\frac12$ times \eqref{eq-itodetterm2help2}, and \eqref{eq-itodetterm3}, to obtain
	\begin{align*}
		D_t&= 1 + \int_{(0,t]} D_{s-} d \trace(L_s) + \frac12 \sum_{\substack{m,n=1\\m\neq n}}^d (\sigma_{(m,m),(n,n)} - \sigma_{(m,n),(n,m)})  \int_{(0,t]} D_{s-} ds + \\
		&\quad + \sum_{0<s\leq t} D_{s-} \left(\det(\Id + \Delta L_s)  - 1 - \trace(\Delta L_s)\right)
	\end{align*}
	since $D_0=\det(X_0)=\det\Id = 1.$ This proves the statement via \eqref{SDE}.	
\end{proof}

Having found out that the determinant process is a one-dimensional stochastic exponential, we can now easily derive its explicit representation as presented in the next corollary.

\begin{corollary} \label{cor-detprocessexplicit}
		Let $L=(L_t)_{t\geq 0}$ be a Lévy process in $\RR^{d\times d}$ fulfilling \eqref{eq-nonsingular} and let $D=(D_t)_{t\geq 0}$ be the determinant process of its stochastic exponential defined in \eqref{eq-detprocess}. Then 
	\begin{align*}
		D_t&= \exp\Bigg(\trace(L_t) -  \frac12 \sum_{m,n=1}^d \sigma_{(m,n),(n,m)} \cdot t  \Bigg) \prod_{0<s\leq t} \det\big((\Id+\Delta L_s) e^{-\Delta L_s}\big), \quad t\geq 0.
	\end{align*}
\end{corollary}
\begin{proof}
	The formula for $D_t$ is an immediate consequence of Theorem~\ref{thm-detprocess} and the explicit formula of the one-dimensional stochastic exponential as given e.g. in  \cite[Thm. II.37]{protter}, and upon noticing that 
	$$\Delta \check{L}=\det(\Id + \Delta L)-1, \quad e^{-\trace(\Delta L)} = \det(e^{-\Delta L}),$$
	and
	\begin{align*}
		[\trace(L^c),\trace(L^c)]_t &= \left[ \sum_{m=1}^d (L^{(m,m)})^c, \sum_{n=1}^d (L^{(n,n)})^c \right]_t = \sum_{m,n=1}^d  \left[ (L^{(m,m)})^c, (L^{(n,n)})^c \right]_t\\
		&= \sum_{m,n=1}^d \sigma_{(m,m),(n,n)} \cdot t .\qedhere
	\end{align*}
\end{proof}

The special linear group $\SL$ is the subgroup of matrices in $\GL$ with determinant~$1$. Interestingly, Corollary~\ref{cor-detprocessexplicit} allows for a direct classification of all Lévy processes whose exponential functional is in $\SL$ that we state as follows.

\begin{corollary}\label{cor-speciallineargroup}
	Let $L=(L_t)_{t\geq 0}$ be a Lévy process in $\RR^{d\times d}$ fulfilling \eqref{eq-nonsingular} and let $X=(X_t)_{t\geq 0}$ denote its stochastic exponential. Then $X_t\in \SL$ for all $t\geq 0$ if and only if
	\begin{enumerate}
		\item the Brownian components $B_{L,t}$ of $L$ have trace $0$, i.e. $\sum_{i=1}^d B_{L,t}^{(i,i)} =0$ a.s. for all $t\geq 0$,
		\item the trace of the drift of $L^c$ fulfils $2 \sum_{i=1}^d (\gamma_L^0)^{(i,i)} =  \sum_{m,n=1}^d \sigma_{(m,n),(n,m)},$ and
		\item the Lévy measure has support $\supp(\nu_L)\subseteq \{a\in \RR^{d\times d} \setminus \{0\} : \det(\Id+a)=1 \}$.
	\end{enumerate}		
\end{corollary}

\subsection{Asymptotic behaviour of the determinant process}\label{sectdetLLN}

It can be expected that the determinant of the stochastic exponential grows exponentially. We therefore start our considerations with the following lemma that determines the structure of the logarithm of $D$. 

\begin{lemma}\label{lem-hatD}
		Let $L=(L_t)_{t\geq 0}$ be a Lévy process in $\RR^{d\times d}$ fulfilling \eqref{eq-nonsingular} and let $D=(D_t)_{t\geq 0}$ be the determinant process of its stochastic exponential defined in \eqref{eq-detprocess}. 
	Then $(\check{D}_t)_{t\geq 0}$ given by
	\begin{align} \label{eq-hatD}
		\check{D}_t:=	\log |D_t| &= \trace(L_t^c) - \frac12 \sum_{m,n=1}^d \sigma_{(m,n),(n,m)} \cdot t + \sum_{0<s\leq t} \log |\det(\Id + \Delta L_s)|, \quad t\geq 0,
	\end{align}
	is an (additive) Lévy process in $\RR$ with characteristic triplet $(\sigma_{\check{D}}, \gamma_{\check{D}}, \nu_{\check{D}})$ given by
	\begin{align*}
		\sigma_{\check{D}}&= \sum_{m,n=1}^d \sigma_{(m,m),(n,n)} \\
		\gamma_{\check{D}}&= \trace(\gamma_L)  - \frac12 \sum_{m,n=1}^d \sigma_{(m,n),(n,m)}  \\
		& \qquad + \int_S \Big( (\log|\det(\Id+x)|) \mathds{1}_{\{|\det(\Id+x)| \in [e^{-1},e]\}}  -\trace(x)   \mathds{1}_{\{\|\vect x \|\leq 1\}}   \Big)\nu_L(dx) \\
		\nu_{\check{D}}&= \Phi(\nu_L)|_{\RR\setminus \{0\}},
	\end{align*}
	where $S=\RR^{d\times d}\setminus \big(\{x \in \RR^{d\times d}: \det(\Id+x) =0\}\cup \{0\} \big)$, and for the measure transform
	$$\Phi: S \to \RR\setminus \{0\}, \quad x\mapsto \log|\det(\Id+x)|.$$
	Moreover, we have $\EE|\check{D}_1|<\infty$ if and only if
	\begin{equation} \label{eq-hatDmeancond}
		\int_S  (\log|\det(\Id+x)|) \mathds{1}_{\{|\det(\Id+x)| \in (0,e^{-1})\cup (e,\infty)\}}\nu_L(dx)<\infty,
	\end{equation}
	and, if this is the case,
	\begin{equation} \label{eq-hatDmean}
		\EE[\check{D}_1] = \trace(\gamma_1)  - \frac12 \sum_{m,n=1}^d \sigma_{(m,n),(n,m)}  + \int_S \Big( \log|\det(\Id+x)|  -\trace(x)   \mathds{1}_{\{|\vect x |\leq 1\}}   \Big)\nu_L(dx).
	\end{equation}
\end{lemma}

\begin{proof}
	The formula for $\check{D}_t$ is immediate from Corollary~\ref{cor-detprocessexplicit}, while the fact that $(\check{D}_t)_{t\geq 0}$ is a Lévy process has been shown in \cite[Lemma 3.4]{BLM}.\\
	To derive the characteristic triplet of $\check{D}$ we rely on the Lévy-It\^o decomposition, cf. \cite[Thm. 19.2]{sato2nd}, and observe first that the Brownian motion components of $\check{D}$ and $L$ satisfy
	\begin{equation} \label{eq-hatDBrownian} B_{\check{D},t} = \trace(B_{L,t})= \sum_{m=1}^d B_{L,t}^{(m,m)},\end{equation}
	from which we obtain
	$$\sigma_{\check{D}}= \var( B_{\check{D},1}) = \sum_{m,n=1}^d \sigma_{(m,m),(n,n)},$$
	in analogy to the computation in the proof of Corollary~\ref{cor-detprocessexplicit}.\\
	Further, from \eqref{eq-hatD} we immediately obtain
	\begin{equation} \label{eq-hatDjump} \Delta \check{D}_t = \log |\det(\Id+\Delta L_t)|, \quad t\geq 0,\end{equation}
	proving the stated form of the Lévy measure of $\check{D}$. Lastly, the Lévy-Itô decomposition of $\check{D}$ yields
	\begin{align*}
		\gamma_{\check{D}} &= \check{D}_1 - B_{\check{D},1} - \lim_{\varepsilon\downarrow0} \Bigg( \sum_{\substack{0<s\leq 1\\ |\Delta \check{D}_s|>\varepsilon}} \Delta \check{D}_s - \int_{ |x|\in (\varepsilon,1]} x \nu_{\check{D}}(dx)  \Bigg) \\
		&= \trace(L_1^c)- B_{\check{D},1} - \frac12 \sum_{m,n=1}^d \sigma_{(m,n),(n,m)}  \\
		&\qquad + \lim_{\varepsilon\downarrow0} \Bigg(  \sum_{\substack{0<s\leq 1\\ |\log |\det(\Id + \Delta L_s)||\leq \varepsilon}} \log |\det(\Id + \Delta L_s)| + \int_{ |x|\in (\varepsilon,1]} x \nu_{\check{D}}(dx) \Bigg),
	\end{align*}
	due to \eqref{eq-hatD} and \eqref{eq-hatDjump}. Setting $E_\varepsilon:=\{x\in S:\, |\log |\det(\Id + x)||\leq \varepsilon \}$ and again applying the Lévy-Itô decomposition (now of $L^c$) as well as \eqref{eq-hatDBrownian} we observe that
	\begin{align*}
		\gamma_{\check{D}} &= \trace(\gamma_L)  - \frac12 \sum_{m,n=1}^d \sigma_{(m,n),(n,m)}  \\
		&\qquad + \lim_{\varepsilon\downarrow0} \Bigg(\sum_{\substack{0<s\leq 1\\  \Delta L_s \in E_\varepsilon}} \log |\det(\Id + \Delta L_s)| + \int_{ |x|\in (\varepsilon,1]} x \nu_{\check{D}}(dx)  - \int_{\|\vect x \|\in(\varepsilon,1]} \trace(x) \nu_L(dx)  \Bigg)\\
		&= \trace(\gamma_L)  - \frac12 \sum_{m,n=1}^d \sigma_{(m,n),(n,m)} + \int_S  (\log|\det(\Id+x)|) \mathds{1}_{\{|\det(\Id+x)|=1\}}   \nu_L(dx)\\
		&\qquad + \int_S \bigg( (\log|\det(\Id+x)|) \mathds{1}_{\{|\det(\Id+x)| \in [e^{-1},e]\}} - \trace(x) \mathds{1}_{\{\|\vect x \|\leq 1\}} \bigg)\nu_L(dx) ,
	\end{align*}
	which gives the stated formula for $\gamma_{\check{D}}$.\\
	The condition \eqref{eq-hatDmeancond} for $\EE|\check{D}_1|<\infty$ now follows directly from the given form of $\nu_{\check{D}}$ and \cite[Ex. 25.12]{sato2nd}, while the mean of $\check{D_1}$ can by computed by \cite[Ex. 25.12]{sato2nd} as
	\begin{align*}
		\EE[\check{D}_1]&= \gamma_{\check{D}} + \int_{|x|>1} x \nu_{\check{D}}(dx)\\
		&= \gamma_{\check{D}} + \int_S \log|\det(\Id+x)| \mathds{1}_{\{|\det(\Id+x)|\in (0,e^{-1})\cup (e,\infty)\}} \nu_L(dx)\\
		&= \trace(\gamma_L)  - \frac12 \sum_{m,n=1}^d \sigma_{(m,n),(n,m)}  + \int_S \Big( \log|\det(\Id+x)|  -\trace(x)   \mathds{1}_{\{\|\vect x \|\leq 1\}}   \Big)\nu_L(dx). \qedhere
	\end{align*}
\end{proof}

The long-time behaviour of one-dimensional, additive Lévy processes is well understood and in particular it is now immediate that $\check{D}$ either drifts to $\pm \infty$ a.s., or oscillates. As the characteristics of $\check{D}$ are known, eclectic related results e.g. from \cite[Sections 36, 37 and 48]{sato2nd} can now be applied to derive LLNs and LILs for $\check{D}$. For the sake of clarity and shortness of our exposition we restrict ourselves here on the following exemplary theorem.

\begin{theorem}\label{thm-SLLN-det}
		Let $L=(L_t)_{t\geq 0}$ be a Lévy process in $\RR^{d\times d}$ fulfilling \eqref{eq-nonsingular} and let $D=(D_t)_{t\geq 0}$ be the determinant process of its stochastic exponential defined in \eqref{eq-detprocess}. 
	\begin{enumerate}
		\item If \eqref{eq-hatDmeancond} holds, then a.s.
		\begin{align*} \lefteqn{\lim_{t\to \infty} t^{-1} \log |D_t|}\\ &= \trace(\gamma_L)  - \frac12 \sum_{m,n=1}^d \sigma_{(m,n),(n,m)}  + \int_S \Big( \log|\det(\Id+x)|  -\trace(x)   \mathds{1}_{\{\|\vect x \|\leq 1\}}   \Big)\nu_L(dx).\end{align*} 
		\item If \eqref{eq-hatDmeancond} fails, then a.s.
		$$\limsup_{t\to \infty} t^{-1} |\log |D_t||=\infty.$$
	\end{enumerate}
\end{theorem}
\begin{proof}
	This follows immediately from \cite[Thm. 36.5]{sato2nd} and Lemma \ref{lem-hatD}.
\end{proof}

\begin{remark}
	Note that condition \eqref{eq-hatDmeancond} holds in particular if $\EE[\log M(\Id+L_1)]<\infty$. This is due to the fact that, since $\det(a)=\lambda_1 \cdots \lambda_d$, with $\lambda_1, \dots, \lambda_d$ being the eigenvalues of $a$ in decreasing order (with respect to absolute value), it holds 
	$$ \det(a) \le \norm{a}^d, \qquad |\det(a)| \ge |\lambda_d|^d = \norm{a^{-1}}^{-d},$$
	and hence $|\log \det(a)| \le d \log M(a)$.  In particular it follows from the approximation in Lemma~\ref{lemma-emeryapprox} that $\E [\log M(X_1)]<\infty$ is a sufficient condition for \eqref{eq-hatDmeancond}.
\end{remark}

Central limit theorems for Lévy processes have been studied in \cite{DoneyMaller02}. The adaptation of their results to the present situation yields the following theorem.

\begin{theorem}\label{thm-CLT-det}
	Let $L=(L_t)_{t\geq 0}$ be a Lévy process in $\RR^{d\times d}$ fulfilling \eqref{eq-nonsingular} and let $D=(D_t)_{t\geq 0}$ be the determinant process of its stochastic exponential defined in \eqref{eq-detprocess}. 
	Set
	$$T_1(x):= \nu_L\big(\{x\in\RR^{d\times d}: |\det(\Id+x)|\in (0,e^{-x})\cup (e^x,\infty)\}\big),$$
	and 
	$$T_2(x):=\nu_L\big(\{x\in\RR^{d\times d}: |\det(\Id+x)|\in  (e^x,\infty)\}\big) - \nu_L\big(\{x\in\RR^{d\times d}: |\det(\Id+x)|\in (0,e^{-x})\}\big). $$
	\begin{enumerate}
		\item Assume $T_1(x)=0$ for some $x>0$. If $\sigma_{\check{D}}= \sum_{m,n=1}^d \sigma_{(m,m),(n,n)}>0$, then 
		$$\frac{\log |D_t| - t \big(\gamma_{\check{D}} + T_2(1)+\int_{(1,\infty)} T_2(y) dy\big)}{\sqrt{\sigma_{\check{D}} \cdot t}} \overset{d}\longrightarrow \mathcal{N}(0,1), \quad t\to \infty.$$
		\item Assume $T_1(x)>0$ for all $x>0$. If
		$$\frac{\sigma_{\check{D}}+\int_{[-x,x]} z^2 \nu_{\check{D}}(dz)}{x^2 T_1(x)} \to \infty, \quad x\to \infty,$$
		then there exist measurable functions $a(t), b(t)>0$ such that
		$$\frac{\log|D_t| - a(t)}{b(t)}  \overset{d}\longrightarrow \mathcal{N}(0,1), \quad t\to \infty.$$
	\end{enumerate}
\end{theorem}
\begin{proof}
	Both statements follow from \cite[Thm. 3.5]{DoneyMaller02} and the derived characteristics of $\check{D}$ in Lemma \ref{lem-hatD}. For  2. note that
	$$T_1(x) = \nu_{\check{D}}((x,\infty)) + \nu_{\check{D}}((-\infty, -x)),$$
	and hence 
	\begin{align*}
		\int_0^x y T_1(y)\, dy &= \int_0^x \int_{(y,\infty)} y \, \nu_{\check{D}}(dz) \,dy + \int_0^x \int_{(-\infty,y)} y\, \nu_{\check{D}}(dz)\, dy \\
		&= \int_{(0,\infty)} \int_0^{x\wedge z} y\, dy \,\nu_{\check{D}}(dz)  + \int_{(-\infty, 0)} \int_0^{x \wedge (-z)} y\, dy \,\nu_{\check{D}}(dz) \\
		&= \frac12 \left( \int_{(0,\infty)} (x\wedge z)^2 \nu_{\check{D}}(dz) + \int_{(-\infty, 0)} (x\wedge(-z))^2 \nu_{\check{D}}(dz) \right)\\
		&= \frac{x^2}{2} T_1(x) + \frac12 \int_{[-x,x]} z^2 \nu(dz).
	\end{align*}
	Thus the condition stated in \cite[Eq.(3.12)]{DoneyMaller02}, namely
	$$\frac{\sigma_{\check{D}} + 2\int_0^x y T_1(y)\, dy}{x^2 T_1(x)} \to \infty, \quad x\to \infty,$$
	is equivalent to the stated condition in 2.
\end{proof}

\appendix
\section{The stochastic exponential as Feller process}\label{appendix:stochastic.exponential.as.Feller.process}

 In this section we prove the formula for the generator of the stochastic exponential given in Proposition~\ref{prop:generatorXt}. We start with a lemma that provides us with a convenient SDE for the vectorised process $\vec{X}=(\vec{X}_t)_{t\geq 0}$ defined via $\vec{X}_t:= \vect(X_t)$, $t\geq 0$. 

\begin{lemma}\label{lem-vecX}
	Let $L=(L_t)_{t\geq 0}$ be a Lévy process in $\RR^{d\times d}$ fulfilling \eqref{eq-nonsingular}, let $X=(X_t)_{t\geq 0}$ be its stochastic exponential and $\vec{X}$ its vectorisation. Then 
	\begin{equation} \label{eq-SDEvect} d\vec{X}_t= \big(\Id \otimes \vect^{-1}(\vec{X}_{t-})\big) d\vec{L}_t,\quad t\geq 0. \end{equation}
\end{lemma}
\begin{proof}
	From \eqref{SDE} we obtain
	\begin{align*}
		\vec{X}_t&= \vect\bigg(\Id+ \int_{(0,t]} X_{s-} dL_s \bigg) = \vect(\Id) + \vect\bigg(\int_{(0,t]} X_{s-} dL_s \bigg)\\
		&= \vect(\Id) + \int_{(0,t]} (\Id \otimes X_{s-}) d(\vect(L_s)) = \vect(\Id) + \int_{(0,t]} \big(\Id\otimes \vect^{-1}(\vec{X}_{s-})\big) d\vec{L}_s,
	\end{align*}
where we have applied \cite[Lemma 2.1(vi)]{Behme2012b} for the equality between the lines. This implies the statement.
\end{proof}

 Now we are ready to prove Proposition \ref{prop:generatorXt}.

\begin{proof}[Proof of Proposition \ref{prop:generatorXt}.]
	We consider the vectorisation  $\vec{X}=(\vec{X}_t)_{t\geq 0}$ of the stochastic exponential and note that due to the definition of the generator as
	$$\mathcal{A}_X f(x) = \lim_{t\downarrow 0} \frac{\EE^x[f(X_t)] - f(x)}{t},$$
	we have $\mathcal{A}_{X} f = \mathcal{A}_{\vec{X}}(f\circ \vect^{-1}).$ We will thus focus on the derivation of $\mathcal{A}_{\vec{X}}$ for which we follow the procedure described in \cite[Section 6.1]{schnurrdiss}. \\
			 According to Lemma \ref{lem-vecX} the process $\vec{X}$ fulfils an SDE of the form
	$$d\vec{X}_t= \Psi(\vec{X}_{t-}) d\vec{L}_t, \quad t\geq 0,$$
	with $\Psi: \RR^{d^2} \to \RR^{d^2\times d^2}: x\mapsto (\Id \otimes \vect^{-1}(x))$ being Lipschitz continuous and of linear growth. Hence we can apply \cite[Thm. 1.1]{kuehn} and obtain that $\vec{X}$ is a rich Feller process if and only if
	\begin{equation}\label{eq-condrichFeller}
		\nu_L\left(\vect^{-1}\big( \{y\in\RR^{d^2}: \Psi(x)y\in B_r(-x)\}\big)\right) \overset{\|x\|\to\infty}\longrightarrow 0, \quad \text{for all }r>0.
	\end{equation}
	Hereby, via \cite[Prop. 7.1.9]{Bernstein},
	\begin{align*}
		A(x)&:= \{y\in\RR^{d^2}: \Psi(x)y\in B_r(-x)\} = \{y\in\RR^{d^2}: \norm{  (\Id \otimes \vect^{-1}(x)) y + x} \leq r \}\\
		&
		 = \{y\in\RR^{d^2}: \norm{  \vect( \vect^{-1}(x) \vect^{-1}(y) +\vect^{-1}(x))} \leq r \}\\
		 &= \{y\in\RR^{d^2}: \norm{  \vect^{-1}(x) (\vect^{-1}(y) + \Id) } \leq r \},
	\end{align*}
	and we note that $\mathds{1}_{A(x)}(y) \overset{\norm{x}\to \infty}\longrightarrow \mathds{1}_{\{y=-\vect \Id\}}$. Thus, slightly extending \cite[Example 4.4]{kuehn}, we observe that as $\nu_L$ is $\sigma$-finite, an application of Lebesgue's dominated convergence yields
	$$	\nu_L (\vect^{-1} A(x)) \overset{\norm{x}\to \infty}\longrightarrow \nu_L(\{-\Id\}) = 0,$$
	due to \eqref{eq-nonsingular}. \\
	Thus \ref{eq-condrichFeller} is fulfilled and, according to \cite[Thm. 1.1]{kuehn}, the symbol $p(x,\xi), x,\xi\in\RR^{d^2}$, of $\vec{X}$ is given by
	\begin{align*} 
		p(x,\xi) &= \psi_L(\Psi(x)^T \xi) = \psi_L\big((\Id\otimes \vect^{-1}(x))^T \xi\big)=: \psi_L(z(x,\xi)),
		\end{align*}
	with $\psi_L$ denoting the Lévy-Khintchine exponent of $\vec{L}$ given by
	\begin{align*}
		\psi_L(z)&= - i \vect(\gamma_L)^T z + \frac12 z^T \Sigma_L z - \int_{y\neq 0} \Big(e^{iz^Ty} - 1- iz^T y \mathds{1}_{\{ \norm{y} \leq 1\}}\Big) \vect(\nu_L(dy)).
	\end{align*}
	Hence
	\begin{align*}
		p(x,\xi) &= - i \vect(\gamma_L)^T (\Id\otimes \vect^{-1}(x))^T \xi + \frac12 \xi^T (\Id\otimes \vect^{-1}(x)) \Sigma_L (\Id\otimes \vect^{-1}(x))^T \xi \\
		& \quad - \int_{y\neq 0} \Big(e^{i y^T (\Id\otimes \vect^{-1}(x))^T \xi} - 1- i y^T (\Id\otimes \vect^{-1}(x))^T \xi \mathds{1}_{\{ \norm{y} \leq 1\}}\Big) \vect(\nu_L(dy))\\
		&= - i \vect(\gamma_L)^T (\Id\otimes \vect^{-1}(x))^T \xi + \frac12 \xi^T (\Id\otimes \vect^{-1}(x)) \Sigma_L (\Id\otimes \vect^{-1}(x))^T \xi \\
		& \quad - \int_{y\neq 0} \Big(e^{i ((\Id\otimes \vect^{-1}(x))y)^T \xi} - 1- i ((\Id\otimes \vect^{-1}(x))y)^T \xi \mathds{1}_{\{ \norm{ (\Id\otimes \vect^{-1}(x))y } \leq 1\}}\Big) \vect(\nu_L(dy))\\
		&\quad - i \int_{y\neq 0}  ((\Id\otimes \vect^{-1}(x))y)^T \left(  \mathds{1}_{\{ \norm{ (\Id\otimes \vect^{-1}(x))y } \leq 1\}}- \mathds{1}_{\{ \norm{y} \leq 1\}} \right) \vect\nu_L(dy)\, \xi\\
		&= - i \vec\ell(x)^T \xi + \frac12 \xi^T \vec{Q}(x) \xi   - \int_{z\neq 0} \Big(e^{i z^T \xi} - 1- i z^T \xi\,  \chi(z) \Big) \vec{N}(x,dz),
	\end{align*}
with the triplet $(\vec\ell(x), \vec{Q}(x), \vec{N}(x,dy))$ \begin{align*}
	\vec{\ell}(x)&= (\Id\otimes \vect^{-1}(x)) \vect(\gamma_L) \\
	&\qquad +  \int_{y\in\RR^{d^2}\setminus\{0\}}  (\Id\otimes \vect^{-1}(x))y  \left(  \mathds{1}_{\{ \norm{ (\Id\otimes \vect^{-1}(x))y } \leq 1\}}- \mathds{1}_{\{ \norm{y} \leq 1\}} \right) \vect\nu_L(dy),  \\
	\vec{Q}(x) &= (\Id\otimes \vect^{-1}(x)) \Sigma_L (\Id\otimes \vect^{-1}(x))^T, \quad \text{and} \\
	\vec{N}(x,dy)&= \vect(\nu_L(\vec{T}_x^{-1}(dy))) = \nu_L(\vect^{-1}(\vec{T}_x^{-1}(dy))),
\end{align*}
being the image measure of $\nu_L$ under the transform $\vec{T}_x\circ \vect$ with
$$\vec{T}_x(y): \RR^{d^2} \to \RR^{d^2} , y\mapsto (\Id\otimes \vect^{-1}(x)) y, \quad x\in\RR^{d^2},$$
and $\chi(z)=\mathds{1}_{\{ \norm{z} \leq 1\}}$ as cut-off function. Thus, as described in \cite[Section 6.1]{schnurrdiss}, the extended generator of $\vec{X}$ is 
\begin{align*}\cA_{\vec{X}} u(x) &= \vec{\ell}(x)^T  \nabla u(x) + \frac12 \sum_{i,j=1}^{d^2} \vec{Q}(x)^{(i,j)} \frac{\partial}{\partial x^{(i)}} \frac{\partial}{\partial x^{(j)}} u(x) \\
	&\quad  + \int_{y\neq 0} \Big( u(x+y) - u(x) -  y^T \nabla u(x)\mathds{1}_{\{\norm{y}\leq 1\}}  \Big) \vec{N}(x, dy),\end{align*}
for $u\in C^2_c(\RR^{d^2})$. \\
Using \cite[Prop. 7.1.9]{Bernstein} we observe that
\begin{align*}
	\vec{\ell}(x)&= \vect(\vect^{-1}(x) \gamma_L) +  \int_{y\in \RR^{d\times d}\setminus\{0\}}  \vect(\vect^{-1}(x) y)  \left(  \mathds{1}_{\{ \norm{ \vect^{-1}(x) y } \leq 1\}}- \mathds{1}_{\{ \norm{y} \leq 1\}} \right) \nu_L(dy),
\end{align*}
and hence
$$\vec\ell(x)^T \nabla u(x) = \sum_{i=1}^{d^2} \ell(x)^{(i)} \frac{\partial}{\partial x^{(i)}} u(x) = \sum_{i,j=1}^d \ell(\vect^{-1}x)^{(i,j)} \frac{\partial}{\partial (\vect^{-1}x)^{(i,j)}} f(\vect^{-1}(x))$$
for $f(z)=u(\vect(z))$ in $C_b^2(\RR^{d\times d})$ and 
$$\ell(z)= z \gamma_L +  \int_{y\in \RR^{d\times d}\setminus\{0\}}  z y  \left(  \mathds{1}_{\{ \norm{ z y } \leq 1\}}- \mathds{1}_{\{ \norm{y} \leq 1\}} \right) \nu_L(dy).$$
Similar considerations yield the form of the integral term. Lastly, likewise
\begin{align*}
	\sum_{i,j=1}^{d^2} \vec{Q}(x)^{(i,j)} \frac{\partial}{\partial x^{(i)}} \frac{\partial}{\partial x^{(j)}} u(x) &= \sum_{i,j,k,l=1}^d \big[(\Id \otimes z) \Sigma_L (\Id\otimes z)^T\big]^{(i+(j-1)d, k+(l-1)d)} \frac{\partial}{\partial z^{(i,j)}} \frac{\partial}{\partial z^{(k,l)}} f(z),
\end{align*}
with $z=\vect^{-1} x$, and
\begin{align*}
\big[(\Id \otimes z) \Sigma_L (\Id\otimes z)^T\big]^{(i+(j-1)d, k+(l-1)d)} &= \sum_{n,m=1}^{d^2} 	(\Id\otimes z)^{(i+(j-1)d,m)} \Sigma_L^{(m,n)} (\Id\otimes z^T)^{(n, k+(l-1)d)}\\
&= \sum_{n,m=1}^{d} 	z^{(i,m)} \Sigma_L^{(m+(j-1)d,n+(l-1)d)} z^{(k,n)}\\
&= \sum_{n,m=1}^{d} 	z^{(i,m)} \sigma_{(m,j),(n,l)} z^{(k,n)}.
\end{align*}
This finishes the proof.
\end{proof}

\end{document}